\documentclass[12pt]{amsart}

\usepackage{fullpage}
\usepackage{amsthm,amssymb,url}
\usepackage[usenames]{color}

\begin{document}

\numberwithin{equation}{section}

\def\comment#1#2{\textcolor{blue}{(#1: #2)}}
\def\red#1{\textcolor{red}{#1}}

\theoremstyle{plain}
\newtheorem{theorem}{Theorem}[section]
\newtheorem{conjecture}[theorem]{Conjecture}
\newtheorem{corollary}[theorem]{Corollary}
\newtheorem{definition}[theorem]{Definition}
\newtheorem{lemma}[theorem]{Lemma}
\newtheorem{proposition}[theorem]{Proposition}
\newtheorem{problem}[theorem]{Problem}

\theoremstyle{definition}
\newtheorem{algorithm}[theorem]{Algorithm}
\newtheorem{construction}[theorem]{Construction}
\newtheorem{example}[theorem]{Example}
\newtheorem{remark}[theorem]{Remark}

\def\ldiv{\backslash}
\def\im#1{\mathrm{Im}(#1)}
\def\ker#1{\mathrm{Ker}(#1)}
\def\aut#1{\mathrm{Aut}(#1)}
\def\aff#1{\mathrm{Aff}(#1)}                        
\def\qaff#1{\mathcal Q_{\mathrm{Aff}}(#1)}          
\def\rmlt#1{\mathrm{RMlt}(#1)}
\def\mlt#1{\mathrm{Mlt}(#1)}
\def\inn#1{\mathrm{Inn}(#1)}
\def\dis#1{\mathrm{Dis}(#1)}
\def\qhom#1{\mathcal Q_{\mathrm{Hom}}(#1)}          
\def\qgal#1{\mathcal Q_{\mathrm{Gal}}(#1)}          
\def\Z{\mathbb Z}
\def\F{\mathbb F}
\def\cq#1{\mathrm{Cjg}(#1)} 
\def\conj#1{\phi_{#1}} 
\def\sym#1{S_{#1}}

\def\gen#1{\langle\, #1 \,\rangle}
\newcommand\Soc{\mathrm{Soc}}
\newcommand\PSL{\mathrm{PSL}}
\newcommand\GL{\mathrm{GL}}
\newcommand\SL{\mathrm{SL}}


\title{Connected quandles and transitive groups}

\author{Alexander Hulpke}

\address[Hulpke]{Department of Mathematics, Colorado State University, 1874 Campus Delivery, Ft. Collins, Colorado 80523, U.S.A.}

\author{David Stanovsk\'y}

\address[Stanovsk\'y]{Department of Algebra, Faculty of Mathematics and Physics, Charles University, Sokolovsk\'a 83, Praha 8, 18675, Czech Republic}

\author{Petr Vojt\v{e}chovsk\'y}

\address[Stanovsk\'y, Vojt\v{e}chovsk\'y]{Department of Mathematics, University of Denver, 2280 S Vine St, Denver, Colorado 80208, U.S.A.}

\email[Hulpke]{hulpke@math.colostate.edu}
\email[Stanovsk\'y]{stanovsk@karlin.mff.cuni.cz}
\email[Vojt\v{e}chovsk\'y]{petr@math.du.edu}

\begin{abstract}
We establish a canonical correspondence between connected quandles and certain configurations in transitive groups, called quandle envelopes. This correspondence allows us to efficiently enumerate connected quandles of small orders, and present new proofs concerning connected quandles of order $p$ and $2p$.
We also present a new characterization of connected quandles that are affine.
\end{abstract}



\thanks{Research partially supported by the Simons Foundation Collaboration
Grant~244502 to Alexander Hulpke, the GA\v CR grant 13-01832S to David Stanovsk\'y, and the Simons Foundation Collaboration
Grant 210176 to Petr Vojt\v{e}chovsk\'y.}

\maketitle

\section{Introduction}

\subsection{Motivation}

Let $Q=(Q,\cdot)$ be a set with a single binary operation. Then $Q$ is a \emph{rack} if all \emph{right translations}
\begin{displaymath}
    R_x:Q\to Q,\quad y\mapsto yx
\end{displaymath}
are automorphisms of $Q$. If the rack $Q$ is idempotent, that is, if $xx=x$ for all $x\in Q$, then $Q$ is a \emph{quandle}.

Consider the \emph{right multiplication group}
\begin{displaymath}
    \rmlt{Q}=\gen{ R_x: x\in Q},
\end{displaymath}
and note that $Q$ is a rack if and only if $\rmlt Q$ is a subgroup of the automorphism group $\aut Q$. A rack $Q$ is said to be \emph{connected} (also \emph{algebraically connected} or \emph{indecomposable}) if $\rmlt Q$ acts transitively on $Q$. The main subject of this work are connected quandles.

\medskip

An important motivation for the study of quandles is the quest for computable invariants of knots and links. Connected quandles are of prime interest here because all colors used in a knot coloring fall into the same orbit of transitivity.

From a broader perspective, quandles are a special type of set-theoretical solutions to the quantum Yang-Baxter equation \cite{Dri,Eis} and can be used to construct Hopf algebras \cite{AG}. There are indications, such as \cite{ESG}, that understanding racks and quandles, particularly the connected ones, is an important step towards understanding general set-theoretical solutions of the Yang-Baxter equation.

\medskip

Our main result, Theorem \ref{Th:Canonical}, is a correspondence between connected quandles and certain configurations in transitive groups. Some variants of this representation were discovered independently in \cite{EGTY,Gal-ldq,Joy,Pie}, but none of these works contains a complete characterization of the configurations as in Theorem \ref{Th:Canonical}, nor a discussion of the isomorphism problem as in Theorem \ref{Th:CanIso}. Using the correspondence, we reprove (and occasionally extend) several known results on connected quandles in a simpler and faster way. We focus on enumeration of ``small'' connected quandles, namely those of order less than $48$ (see Section \ref{Sc:Small} and Algorithm \ref{Alg:Enum}) and those with $p$ or $2p$ elements (see Section \ref{Sc:P}). Our proof of non-existence of connected quandles with $2p$ elements, for any prime $p>5$, is based on a new group-theoretical result for transitive groups of degree $2p$, Theorem \ref{thmperm}.

\medskip
The modern theory of quandles originated with Joyce's paper \cite{Joy} and the introduction of the knot quandle, a complete invariant of oriented knots. Subsequently, quandles have been used as the basis of various knot invariants \cite{Car,CJKLS,CESY} and in algorithms on knot recognition \cite{CESY,FLS}.

But the roots of quandle theory are much older, going back to self-distributive quasigroups, or \emph{latin quandles} in today's terminology, see \cite{Sta-latin} for a comprehensive survey of results on latin quandles and their relation to the modern theory. Another vein of results has been motivated by the abstract properties of reflections on differentiable manifolds \cite{Kik,Loos}, resulting in what is now called \emph{involutory quandles} \cite{Sta-involutory}. Yet another source of historical examples is furnished by conjugation in groups, which eventually led to the discovery of the above-mentioned knot quandle by Joyce and Matveev \cite{Joy,Mat}.

\medskip

Quandles have also been studied as algebraic objects in their own right, and we will now briefly summarize the most relevant results. Every quandle decomposes into orbits of transitivity of the natural action of its right multiplication group. An attempt to understand the orbit decomposition was made in \cite{EGTY,NW}, and a full description has been obtained in two special cases: for medial quandles \cite{JPSZ} and for involutory quandles \cite{Pie}. The orbits are not necessarily connected, but they share certain properties with connected quandles.

There have been several attempts to understand the structure of connected quandles, see e.g. \cite{AG}. In our opinion, the homogeneous representation reviewed in Section \ref{Sc:Homogeneous} is most useful in this regard. It was introduced by Galkin and Joyce \cite{Gal-ldq,Joy}, and led to several structural and enumeration results, such as \cite{ESG,Gal-survey,Ven}. Some of them will be presented in Sections~\ref{Sc:Small} and~\ref{Sc:P}. A classification of simple quandles can be found in~\cite{AG,Joy-simple}.

\subsection{Summary of results}

The paper is written as a self-contained introduction to connected quandles. Therefore, in the next two sections, we review the theory necessary for proving the main result. Although the opening sections contain no original ideas, our presentation is substantially different from other sources. We prove the main result in Section \ref{Sc:Canonical}, and the rest of the paper is concerned with its applications.

In Section \ref{Sc:Basic} we develop basic properties of quandles in relation to the right multiplication group and its derived subgroup. In Section \ref{Sc:Homogeneous} we introduce the homogeneous representation (Construction \ref{Co:Homogeneous}) and characterize homogeneous quandles as precisely those obtained by this construction (Theorem \ref{Th:Homogeneous}). In Section \ref{Sc:Minimal} we discuss homogenous representations that are minimal with respect to the underlying group (Theorem \ref{Th:Minimal}).

In Section \ref{Sc:Canonical} we prove the main result (Theorem \ref{Th:Canonical}), a canonical correspondence between connected quandles and quandle envelopes. We also describe all isomorphisms between two connected quandles in the canonical representation (Lemma \ref{Lm:CanIso}). As a consequence, we solve the isomorphism problem (Theorem \ref{Th:CanIso}) and describe the automorphism group (Proposition \ref{Pr:Aut}).

Then we focus on two particular classes of connected quandles. In Section \ref{Sc:Latin} we characterize latin quandles in terms of their homogeneous and canonical representations (Propositions \ref{Pr:Latin1} and \ref{Pr:Latin2}). Section \ref{Sc:Affine} contains a characterization of connected affine quandles (Theorem \ref{Th:Affine}): we show that a connected quandle is affine if and only if it is medial if and only if its right multiplication group is metabelian.

The rest of the paper is devoted to enumeration. In Section \ref{Sc:Small} we present an algorithm for enumeration of connected quandles, which is similar to but several orders of magnitude faster than the recent algorithm of Vendramin \cite{Ven}. In addition, using combinatorial and geometric methods, we construct several families of connected quandles, relying on Theorem \ref{Th:Canonical} for a simple verification of connectedness.

In Section \ref{Sc:P} we investigate quandles of size $p$, $p^2$ and $2p$, where $p$ is a prime, using again the correspondence of Theorem \ref{Th:Canonical}. First, we show that any connected quandle of prime power order has a solvable right multiplication group (Proposition \ref{Pr:Solvable}). Then we give a new and conceptually simple proof that every connected quandle of order $p$ is affine. (This has been proved already in \cite{ESG} and, likewise, our proof relies on a deep result of Kazarin about conjugacy classes of prime power order.) Finally, we show in Theorem \ref{thmperm} that transitive groups of order $2p$, $p>5$ cannot contain certain configurations that are necessary for the existence of quandle envelopes. As a consequence, we deduce that there are no connected quandles of order $2p$, $p>5$, a result obtained already by McCarron \cite{McC} by means of Cayley-like representations.

\subsection{Terminology and notation}

Quandles have been rediscovered in several disguises 
and the terminology therefore varies greatly. For the most part we keep the modern quandle terminology that emerged over the last 15 years. However, in some cases we use the older and more general terminology for binary systems developed to a great extent by R.~H.~Bruck in his 1958 book \cite{Bru}. Bruck's terminology is used fairly consistently in universal algebra, semigroup theory, loop theory and other branches of algebra. For instance, we speak of ``right translations'' rather than ``inner mappings.''

\medskip
Every quandle is \emph{right distributive}, i.e., it satisfies the identity $(yz)x=(yx)(zx)$, expressing the fact that $R_x$ is an endomorphism.
A quandle is called \emph{medial} if it satisfies the identity $(xy)(uv)=(xu)(yv)$.

We apply all mappings to the right of their arguments, written as a superscript. Thus $x^\alpha$ means $\alpha$ evaluated at $x$. To save parentheses, we use $x^{\alpha\beta}$ to mean $(x^\alpha)^\beta$, while $x^{\alpha^\beta}$ stands for $x^{(\alpha^\beta)}$.

Let $G$ be a group. For $y\in G$ we denote by $\phi_y$ the conjugation map by $y$, that is, $x^{\phi_y} =y^{-1}xy$ for all $x\in G$. As usual, we use the shorthand $x^y$ instead of $x^{\phi_y}$, and we let $[x,y] = x^{-1}x^y$. Since $(x^{-1})^y = (x^y)^{-1}$, we denote both of these elements by $x^{-y}$.

For $\alpha\in\aut{G}$ we let $C_G(\alpha) = \{z\in G\,:\,z^\alpha=z\}$ be the centralizer of $\alpha$. We write $C_G(x)$ for $C_G(\conj x)$.

If $G$ acts on $X$ and $x\in X$, we let $G_x = \{g\in G\,:\,x^g=x\}$ be the stabilizer of $x$, and $x^G = \{x^g\,:\,g\in G\}$ the orbit of $x$.

Note that for any binary system $(Q,\cdot)$, $a\in Q$ and $\alpha\in\aut Q$, the mapping $R_a^{\,\alpha}$ is equal to $R_{a^\alpha}$, because for every $x\in Q$ we have
\begin{equation}\label{Eq:RConj}
    x^{R_a^{\,\alpha}} = x^{\alpha^{-1}R_a\alpha} = (x^{\alpha^{-1}}\cdot a)^\alpha =  x\cdot a^\alpha = x^{R_{a^\alpha}}.
\end{equation}
Consequently, if $R_a$ is a permutation, then $R_a^{-\alpha} = (R_a^{\,\alpha})^{-1} = R_{a^\alpha}^{-1}$. We will usually use this observation freely, without an explicit reference to \eqref{Eq:RConj}.

\section{The group of displacements}\label{Sc:Basic}

In this section we present basic properties of a certain subgroup of the right multiplication group, called the \emph{group of displacements} (or the \emph{transvection group}). Nearly all results proved in this section can be found in \cite[Section 5]{Joy} or \cite[Section 1]{Joy-simple}, often without a proof. The only fact we were not able to find elsewhere is Proposition \ref{Pr:DQ}(iv). Note that most results here apply to general racks, too.

For a rack $Q$, define the \emph{group of displacements} as
\begin{displaymath}
    \dis Q=\gen{ R_a^{-1}R_b\,:\,a,\,b\in Q}.
\end{displaymath}
Note that
\begin{displaymath}
    \rmlt Q'\leq \dis Q\leq\rmlt Q\leq\aut Q.
\end{displaymath}
The first inequality follows from \eqref{Eq:RConj}, as $[R_a,R_b] = R_a^{-1}R_a^{\,R_b}= R_a^{-1}R_{ab}$ for every $a,b\in Q$.
We also have $R_aR_b^{-1}\in\dis Q$ for every $a,b\in Q$, as $R_aR_b^{-1} = R_b^{-1}R_a^{\,R_b^{-1}} = R_b^{-1}R_c$, where $c = a^{R_b^{-1}}$.

\begin{proposition}\label{Pr:DQ}
Let $Q$ be a rack. Then:
\begin{enumerate}
\item[(i)] $\dis Q\unlhd\aut Q$ and $\rmlt Q\unlhd\aut Q$.
\item[(ii)] The group $\rmlt{Q}/\dis{Q}$ is cyclic.
\item[(iii)] $\dis Q = \{R_{a_1}^{k_1}\dots R_{a_n}^{k_n}\,:\,n\ge 0$, $a_i\in Q$ and $\sum_{i=1}^n k_i=0\}$.
\item[(iv)] If $Q$ is a quandle, the natural actions of $\rmlt{Q}$ and $\dis{Q}$ on $Q$ have the same orbits.
\end{enumerate}
\end{proposition}

\begin{proof}
Let $G=\rmlt{Q}$ and $D=\dis{Q}$.

(i) By \eqref{Eq:RConj}, conjugating a right translation by an automorphism yields another right translation. Thus the generators of both $G$ and $D$ are closed under conjugation in $\aut{Q}$.

(ii) Fix $e\in Q$ and note that $DR_a = DR_e$ for every $a\in Q$. Given an element $\alpha=R_{a_1}^{k_1}\dots R_{a_n}^{k_n}\in G$, we then have $D\alpha = DR_e^{k_1+\cdots+k_n}$, proving that $G/D = \gen{ DR_e}$.

(iii) Let $S$ be the set in question. Since the defining generators of $D$ belong to $S$, and since $S$ is easily seen to be a subgroup of $G$, we have $D\le S$. For the other inclusion, we note that every $\alpha\in S$ can be written as $R_{a_1}^{k_1}\dots R_{a_n}^{k_n}$, where not only $\sum_i k_i=0$ but also $k_i=\pm 1$. Assuming such a decomposition, we prove by induction on $n$ that $\alpha\in D$.

If $n=0$ then $\alpha=1$, the case $n=1$ does not occur, and if $n=2$, we have either $\alpha=R_aR_b^{-1}$ or $\alpha=R_a^{-1}R_b$, both in $D$. Suppose that $n>2$.

If $k_1=k_n$ then there is $1<m<n$ such that $\sum_{i<m}k_i=0$ and $\sum_{i\geq m}k_i=0$. Let $\beta=R_{a_1}^{k_1}\ldots R_{a_{m-1}}^{k_{m-1}}$ and $\gamma=R_{a_m}^{k_m}\ldots R_{a_n}^{k_n}$. Then $\beta$, $\gamma\in D$, and so $\alpha=\beta\gamma\in D$.

If $k_1\ne k_n$ then $\alpha=R_a^k\beta R_b^{-k}$ for some $a$, $b\in Q$, $k= \pm1$ and $\beta=R_{a_2}^{k_2}\ldots R_{a_{n-1}}^{k_{n-1}}$. Note that $\sum_{2\leq i\leq n-1}k_i=0$, hence $\beta\in D$. We have $\alpha = \beta(R_a^k)^\beta R_b^{-k} = \beta R_{a^\beta}^kR_b^{-k}$, and since $R_{a^\beta}^kR_b^{-k}\in D$, we are done.

(iv) Let $\alpha = R_{a_1}^{k_1}\dots R_{a_n}^{k_n}\in G$ and put $k=k_1+\cdots+k_n$. Let $x$, $y\in Q$ be such that $x^\alpha = y$. By (iii), we have $\beta = \alpha R_y^{-k}\in D$ and $x^\beta = x^{\alpha R_y^{-k}} = y^{R_y^{-k}} = y$, using idempotence in the last step.
\end{proof}

The orbits of transitivity of the group $\rmlt Q$ (or, equivalently, of the group $\dis Q$) in its natural action on $Q$ will be referred to simply as \emph{the orbits of $Q$}. Given $e\in Q$, we denote by $e^Q$ the orbit containing $e$. Orbits are subquandles, not necessarily connected.

\begin{example}\label{Ex:3}
In general, the proper inclusion $\rmlt Q' < \dis Q$ can occur in quandles. The smallest example has three elements and two orbits, and is defined by the following Cayley table:
\begin{displaymath}
    \begin{array}{c|ccc}
        Q & 1 & 2 & 3 \\\hline
        1 & 1 & 1 & 1 \\
        2 & 3 & 2 & 2 \\
        3 & 2 & 3 & 3 \\
    \end{array}
\end{displaymath}
\end{example}

However, in connected racks, the equality $\rmlt Q' = \dis{Q}$ always holds.

\begin{proposition}\label{Pr:DQConnected}
If $Q$ is a connected rack then $\rmlt{Q}'=\dis{Q}$.
\end{proposition}

\begin{proof}
It remains to prove that every generator $R_a^{-1}R_b$ of $\dis{Q}$ belongs to $\rmlt{Q}'$.
Let $\alpha\in\rmlt{Q}$ be such that $b=a^\alpha$. Then
$R_a^{-1}R_b = R_a^{-1}R_{a^\alpha} = R_a^{-1}R_a^{\alpha} = [R_a,\alpha]\in\rmlt{Q}'$.
\end{proof}

In some cases, the structure of $\dis Q$ corresponds nicely to the algebraic properties of $Q$. For instance, the following characterization of mediality can be traced back to \cite{Nob}.

\begin{proposition}\label{Pr:DQ2}
Let $Q$ be a rack. Then:
\begin{enumerate}
	\item[(i)] $\dis Q$ is trivial if and only if the multiplication in $Q$ does not depend on the second argument (in quandles, this is equivalent to the multiplication being the left projection).
	\item[(ii)] $\dis Q$ is abelian if and only if $Q$ is medial.
\end{enumerate}
\end{proposition}

\begin{proof}
(i) An inspection of the generating set shows that $\dis Q$ is trivial iff $R_a=R_b$ for every $a$, $b\in Q$. If $Q$ is a quandle, we then get $ab=a^{R_b} = a^{R_a} = a$.

(ii) Note that the following identities are equivalent: $Q$ is medial, $R_yR_{uv} = R_uR_{yv}$, $R_yR_v^{-1}R_uR_v = R_uR_v^{-1}R_yR_v$,
\begin{equation}\label{Eq:DQ2}
    R_yR_v^{-1}R_u=R_uR_v^{-1}R_y.
\end{equation}

Suppose that $\dis{Q}$ is abelian. Then $(R_yR_v^{-1})(R_uR_y^{-1}) = (R_uR_y^{-1})(R_yR_v^{-1}) = R_uR_v^{-1}$, which yields \eqref{Eq:DQ2} upon applying $R_y$ to both sides. Hence $Q$ is medial.

Conversely, if $Q$ is medial, then \eqref{Eq:DQ2} holds, and its inverse yields $R_y^{-1}R_vR_u^{-1}=R_u^{-1}R_vR_y^{-1}$, so $R_xR_y^{-1}R_vR_u^{-1} = R_xR_u^{-1}R_vR_y^{-1} = R_vR_u^{-1}R_xR_y^{-1}$, where we have again used \eqref{Eq:DQ2} in the last equality. Hence $\dis{Q}$ is abelian.
\end{proof}

A prototypical example of medial quandles is the following construction.

\begin{example}\label{Ex:Affine}
Let $A=(A,+)$ be an abelian group and $f\in\aut A$. Define the \emph{affine} quandle (also called \emph{Alexander} quandle) as
\begin{displaymath}
    \qaff{A,f}=(A,*),\quad x*y=x^f + y^{1-f}.
\end{displaymath}
A straightforward calculation shows that $(A,*)$ is indeed a quandle. For mediality, observe that
\begin{displaymath}
    (x*y)*(u*v) = (x^f+y^{1-f})*(u^f+v^{1-f}) = x^{f^2} + y^{(1-f)f} + u^{f(1-f)}+v^{(1-f)^2}
\end{displaymath}
is invariant under the interchange of $y$ and $u$.
\end{example}

Alternatively, given an $R$-module $M$ and an invertible element $r\in R$, then $(M,*)$ with
\begin{displaymath}
    x*y=xr+y(1-r)
\end{displaymath}
is an affine quandle, namely $\qaff{A,f}$ with $A=(M,+)$ and $x^f=xr$. The two definitions are equivalent, and without loss of generality, we can consider $R=\Z[t,t^{-1}]$, the ring of integral Laurent series, and $r=t$.

Most affine quandles are not connected, and most medial quandles are not affine (e.g. the one in Example \ref{Ex:3}). However, we prove later that all connected medial quandles are affine. See \cite{Hou} for comprehensive results on affine quandles.

\section{Homogeneous quandles}\label{Sc:Homogeneous}

An algebraic structure $Q$ is called \emph{homogeneous} if the automorphism group $\aut{Q}$ acts transitively on $Q$. Connected quandles are homogeneous by definition, since their right multiplication group is a transitive subgroup of the automorphism group. Not every quandle is homogeneous, as witnessed by the quandle in Example \ref{Ex:3}.

We will now present a well-known construction of homogeneous quandles. Despite some effort, we were not able to trace its origin. It was certainly used by Galkin \cite{Gal-ldq}, who recognized its importance for representing latin quandles, and also by Joyce \cite{Joy} and others in the context of connected quandles. But the construction seems to be much older, see Loos \cite{Loos}, for instance.

Our immediate goal is to prove Joyce's observation that a quandle $Q$ is homogeneous if and only if it is isomorphic to a quandle obtained by Construction \ref{Co:Homogeneous}.

\begin{construction}\label{Co:Homogeneous}
Let $G$ be a group, $f\in\aut G$ and $H\le C_G(f)$. Denote by $G/H$ the set of right cosets $\{Hx\,:\,x\in G\}$. Define
\begin{displaymath}
    \qhom{G,H,f} = (G/H,*),\quad Hx*Hy = H(xy^{-1})^fy.
\end{displaymath}
\end{construction}

\begin{lemma}\label{Lm:Homogeneous}
Let $Q=\qhom{G,H,f}$ be as in Construction \ref{Co:Homogeneous}. Then $Q$ is a homogeneous quandle.
\end{lemma}
\begin{proof}
First we note that the operation $*$ is well defined. Indeed, if $Hx=Hu$ and $Hy=Hv$ then $u=hx$, $v=ky$ for some $h$, $k\in H$, and
\begin{align*}
    H(uv^{-1})^fv &= H(hxy^{-1}k^{-1})^fky = Hh^f(xy^{-1})^f(k^{-1})^fky\\
        &= Hh(xy^{-1})^fk^{-1}ky = H(xy^{-1})^fy,
\end{align*}
using $H\le C_G(f)$.
Idempotence is immediate from $Hx*Hx = H(xx^{-1})^fx= Hx$. For right distributivity we calculate
\begin{align*}
    (Hx*Hz)*(Hy*Hz) &= H(xz^{-1})^fz*H(yz^{-1})^fz = H[(xz^{-1})^fz((yz^{-1})^fz)^{-1}]^f(yz^{-1})^fz\\
        &= H(xy^{-1})^{f^2}(yz^{-1})^fz = H(xy^{-1})^fy*Hz = (Hx*Hy)*Hz.
\end{align*}
To check that all right translations of $Q$ are permutations of $G/H$, note that for $x$, $y$, $z\in G$ we have
\[ Hx*Hy=Hz\ \Leftrightarrow\ H(xy^{-1})^fy = Hz \ \Leftrightarrow\ Hx^f=Hzy^{-1}y^f\ \Leftrightarrow\ Hx = H(zy^{-1})^{f^{-1}}y, \]
where in the last step we applied $f^{-1}$ to both sides and used $H\le C_G(f)$. Hence, given $Hy$, $Hz$, the equation $Hx*Hy=Hz$ has a unique solution $Hx$.

To prove homogeneity, consider for any $a\in G$ the bijection $\varphi_a:Q\to Q$, $Hx\mapsto Hxa$. Since
\[(Hx)^{\varphi_a} * (Hy)^{\varphi_a} = Hxa * Hya = H(xaa^{-1}y^{-1})^f ya = H(xy^{-1})^f ya = (Hx*Hy)^{\varphi_a},\]
$\varphi_a$ is an automorphism of $Q$. For any $Hx$, $Hy$ there is $a\in Q$ such that $(Hx)^{\varphi_a} = Hxa = Hy$, so $\aut{Q}$ acts transitively on $Q$.
\end{proof}

\begin{example}
Affine quandles are homogeneous. Indeed, if $(A,+)$ is an abelian group and $f\in\aut A$, then $\qaff{A,f}=\qhom{A,0,f}$.
\end{example}

\begin{example}
Knot quandles are homogeneous. Let $K$ be a knot, and let $G_K=\pi_1(U_K)$ be the knot group, where $U_K$ is the complement of a tubular neighborhood of $K$. Let $H_K$ be the peripheral subgroup of $G_K$ and $f_K$ the conjugation by the meridian. Then $\qhom{G_K,H_K,f_K}$ is the knot quandle of $K$. See \cite[Corollary 16.2]{Joy} or \cite[Proposition 2]{Mat} for details.
\end{example}

In the special case of $\qhom{G,H,f}$ where $G$ is a permutation group on a set $Q$ and $H=G_e$ for some $e\in Q$, we define the mapping
\begin{equation}\label{Eq:Pie}
    \pi_e:\qhom{G,G_e,f}\to e^G,\quad G_e\alpha\mapsto e^\alpha.
\end{equation}
Since $G_e\alpha = G_e\beta$ holds if and only if $e^\alpha=e^\beta$, the mapping $\pi_e$ is well defined and bijective.

\begin{proposition}\label{Pr:HomogeneousRepresentation}
Let $Q$ be a quandle and $e\in Q$. Let $G$ be a normal subgroup of $\aut{Q}$, and let $f$ be the restriction of the conjugation by $R_e$ in $\aut{Q}$ to $G$. Then $\qhom{G,G_e,f}$ is well defined and isomorphic to the subquandle $e^G$.
\end{proposition}

\begin{proof}
Since $f$ is a restriction of the conjugation by $R_e\in\rmlt{Q}\le\aut{Q}$ to a normal subgroup $G$ of $\aut{Q}$, it is indeed an automorphism of $G$. To check $G_e\le C_G(f)$, consider $\alpha\in G_e$. For every $x\in Q$ we have $x^{\alpha R_e} = x^\alpha\cdot e = x^\alpha\cdot e^\alpha = (xe)^\alpha = x^{R_e\alpha}$ and so $\alpha^{R_e} = \alpha$ as required. The quandle $\qhom{G,G_e,f}$ is therefore well defined, with multiplication
\begin{displaymath}
    G_e\alpha*G_e\beta = G_e(\alpha\beta^{-1})^f\beta = G_eR_e^{-1}\alpha\beta^{-1}R_e\beta = G_e\alpha R_e^\beta.
\end{displaymath}
The bijective mapping $\pi_e$ from \eqref{Eq:Pie} is an isomorphism $\qhom{G,G_e,f}\to e^G$, since
\begin{displaymath}
    (G_e\alpha*G_e\beta)^{\pi_e} = e^{R_e^{-1}\alpha\beta^{-1}R_e\beta} = (e^{\alpha\beta^{-1}}\cdot e)^\beta = e^\alpha\cdot e^\beta = (G_e\alpha)^{\pi_e}\cdot (G_e\beta)^{\pi_e},
\end{displaymath}
where we have used $\beta\in\aut Q$.
\end{proof}

Consider a situation from Proposition \ref{Pr:HomogeneousRepresentation} in which $G$ acts transitively on $Q$. Then
\begin{displaymath}
    e^G = Q\simeq \qhom{G,G_e,f},
\end{displaymath}
and we will call the isomorphism a \emph{homogeneous representation} of $Q$. The most obvious choice $G=\aut Q$ results in the following characterization.

\begin{theorem}[{\cite[Theorem 7.1]{Joy}}]\label{Th:Homogeneous}
A quandle is homogeneous if and only if it is isomorphic to a quandle obtained by Construction \ref{Co:Homogeneous}.
\end{theorem}
\begin{proof}
Lemma \ref{Lm:Homogeneous} establishes the converse implication. For the direct implication, suppose that $Q$ is homogeneous, take $G=\aut{Q}$, and apply Proposition \ref{Pr:HomogeneousRepresentation}.
\end{proof}

In view of Proposition \ref{Pr:DQ}(iv), connected quandles can be represented using $G=\rmlt Q$ or $G=\dis Q$. The two cases will be studied in detail in the next two sections, resulting in the \emph{canonical} and \emph{minimal} representations.

\section{Minimal representation for connected quandles}\label{Sc:Minimal}

Suppose that $Q$ is a connected quandle, $e\in Q$, and let $G=\rmlt Q'=\dis Q$. The homogeneous representation $Q\simeq \qhom{G,G_e,f}$ of Proposition \ref{Pr:HomogeneousRepresentation} will be called \emph{minimal}. The following result (essentially Galkin's \cite[Theorem 4.4]{Gal-ldq}) gives the reason for the terminology.

\begin{theorem}\label{Th:Minimal}
Let $Q$ be a connected quandle. If $Q\simeq\qhom{G,H,f}$ for some group $G$, $f\in\aut G$ and $H\le C_G(f)$, then $\rmlt Q'$ embeds into a quotient of $G$.
\end{theorem}

\begin{proof}
Assume for simplicity that $Q=\qhom{G,H,f}$. Define $\varphi:G\to\aut Q$ by $a\mapsto \varphi_a$, where $(Hx)^{\varphi_a}=Hxa$ as in the proof of Lemma \ref{Lm:Homogeneous}. The mapping $\varphi$ is obviously a homomorphism. We show that $\rmlt Q'$ is a subgroup of $\im\varphi$, and hence that $\rmlt Q'$ embeds into $G/\mathrm{Ker}(\varphi)$.

By Proposition \ref{Pr:DQConnected}, $\rmlt{Q}' = \dis{Q}$. It therefore suffices to check that $R_{Hx}^{-1}R_{Hy}\in\im\varphi$ for every $x$, $y\in G$. Recall that the unique solution to $Hx*Hy=Hz$ is $Hx=H(zy^{-1})^{f^{-1}}y=(Hz)^{R_{Hy}^{-1}}$. Hence for every $x$, $y$, $u\in G$ we have
\begin{displaymath}
    (Hu)^{R_{Hx}^{-1}R_{Hy}} = (H(ux^{-1})^{f^{-1}}x)^{R_{Hy}} = H((ux^{-1})^{f^{-1}}xy^{-1})^fy = Hux^{-1}(xy^{-1})^fy,
\end{displaymath}
proving $R_{Hx}^{-1}R_{Hy} = \varphi_{x^{-1}(xy^{-1})^fy}$.
\end{proof}

In particular, if $Q$ is a finite connected quandle, and if $G$ is of smallest order among all groups such that $Q\simeq \qhom{G,H,f}$, then $G\simeq\rmlt{Q}'$.

\section{Canonical correspondence for connected quandles}\label{Sc:Canonical}

Throughout this section, fix a set $Q$ and an element $e\in Q$. We proceed to establish a one-to-one correspondence between connected quandles defined on $Q$ and certain configurations in transitive groups on $Q$ that we will call quandle envelopes. To distinguish quandles defined on $Q$ from the underlying set $Q$, we will explicitly name the quandle operation on $Q$.

A \emph{quandle folder} is a pair $(G,\zeta)$ such that $G$ is a transitive group on $Q$ and $\zeta\in Z(G_e)$, the center of the stabilizer of $e$.
A~\emph{quandle envelope} is a quandle folder such that $\gen{\zeta^G} = G$, that is, the smallest normal subgroup of $G$ containing $\zeta$ is all of $G$.

For a connected quandle $(Q,\cdot)$, define
\begin{displaymath}
    \mathcal E(Q,\cdot) = (\rmlt{Q,\cdot},R_e).
\end{displaymath}

\begin{lemma}\label{Lm:E}
Let $(Q,\cdot)$ be a connected quandle and $e\in Q$. Then $\mathcal E(Q,\cdot)$ is a quandle envelope.
\end{lemma}
\begin{proof}
Let $(Q,\cdot)$ and $G=\rmlt{Q,\cdot}$. Note that $R_e\in G_e$. For any $\alpha\in G_e\le\aut{Q,\cdot}$, we calculate $x^{\alpha R_e} = x^\alpha\cdot e = x^\alpha\cdot e^\alpha = (xe)^\alpha = x^{R_e\alpha}$, so $R_e\in Z(G_e)$. Since the quandle $(Q,\cdot)$ is connected, $G$ acts transitively on the set $Q$, and for every $x\in Q$ there is $\widehat x\in G$ such that $e^{\widehat x}=x$. Then $R_e^{\,\widehat x} = R_{e^{\widehat x}} = R_x$, proving that $\gen{ R_e^{\,G} } = G$.
\end{proof}

For a quandle folder $(G,\zeta)$, define
\begin{displaymath}
\mathcal Q(G,\zeta) = (Q,\circ),\quad x\circ y = x^{\zeta^{\widehat y}},
\end{displaymath}
where $\widehat y$ is any element of $G$ satisfying $e^{\widehat y}=y$. We shall see that the operation does not depend on the choice of the permutations $\widehat y$, and that $\mathcal Q(G,\zeta)$ is a homogeneous quandle.

\begin{lemma}\label{Lm:Q}
Let $(G,\zeta)$ be a quandle folder on a set $Q$ with a fixed element $e\in Q$. Then:
\begin{enumerate}
\item[(i)] If $\alpha$, $\beta\in G$ satisfy $e^\alpha=e^\beta$ then $\zeta^\alpha = \zeta^\beta$.
\item[(ii)] The definition of $\mathcal Q(G,\zeta)$ does not depend on the choice of the permutations $\widehat y$.
\item[(iii)] The mapping $\pi_e$ of \eqref{Eq:Pie} is an isomorphism of $\qhom{G,G_e,\conj{\zeta}}$ onto $\mathcal Q(G,\zeta)$.
\item[(iv)] $\mathcal Q(G,\zeta)$ is a homogeneous quandle.
\item[(v)] $\rmlt{\mathcal Q(G,\zeta)} = \gen{ \zeta^{\widehat y}\,:\,y\in Q} = \gen{\zeta^G}$.
\item[(vi)] If $(G,\zeta)$ is a quandle envelope, then $\mathcal Q(G,\zeta)$ is a connected quandle.
\end{enumerate}
\end{lemma}
\begin{proof}
For $\alpha$, $\beta\in G$, note that $\zeta^\alpha = \zeta^\beta$ iff $\beta^{-1}\alpha$ commutes with $\zeta$. The latter condition certainly holds when $e^\alpha=e^\beta$ because $\zeta\in Z(G_e)$. This proves (i), and part (ii) follows.

Consider again the bijection $\pi_e$ of \eqref{Eq:Pie}. Since $G$ is transitive, $\pi_e$ is onto $Q$. To check that $\pi_e$ is a homomorphism, note that $\zeta^\beta = \zeta^{\widehat {e^\beta}}$ by (i). Therefore, with $\qhom{G,G_e,\conj{\zeta}} = (G/G_e,*)$, we have
$G_e\alpha*G_e\beta = G_e(\alpha\beta^{-1})^\zeta\beta = G_e\zeta^{-1}\alpha\zeta^{\beta} = G_e\alpha\zeta^\beta$, and thus
\begin{displaymath}
    (G_e\alpha*G_e\beta)^{\pi_e} = (G_e\alpha\zeta^\beta)^{\pi_e} = e^{\alpha\zeta^\beta} = (e^\alpha)^{\zeta^{\widehat{e^\beta}}} = e^\alpha\circ e^\beta = (G_e\alpha)^{\pi_e} \circ (G_e\beta)^{\pi_e}.
\end{displaymath}
This proves (iii), and part (iv) follows from Lemma \ref{Lm:Homogeneous}.

For (v), note that the right translation by $y$ in $(Q,\circ)$ is the mapping $\zeta^{\widehat y}$ and, once again, $\zeta^\beta = \zeta^{\widehat{e^\beta}}$ for any $\beta\in G$. Part (vi) follows.
\end{proof}

\begin{theorem}[Canonical correspondence]\label{Th:Canonical}
Let $Q$ be a set with a fixed element $e\in Q$. Then the mappings
\begin{align*}
    &\mathcal E: (Q,\cdot)\mapsto (\rmlt{Q,\cdot},R_e),\\
    &\mathcal Q: (G,\zeta)\mapsto (Q,\circ),\quad x\circ y = x^{\zeta^{\widehat y}}
\end{align*}
are mutually inverse bijections between the set of connected quandles and the set of quandle envelopes on $Q$.
\end{theorem}
\begin{proof}
In view of Lemmas \ref{Lm:E} and \ref{Lm:Q}, it remains to show that the two mappings are mutually inverse. Let $(G,\zeta)$ be a quandle envelope, and let $(Q,\circ)=\mathcal Q(G,\zeta)$ be the corresponding connected quandle. Then $\rmlt{Q,\circ} = \gen{\zeta^G} = G$ by Lemma \ref{Lm:Q}. Moreover, $x^{R_e}=x\circ e = x^{\zeta^{\widehat e}} = x^\zeta$ thanks to $\widehat e\in G_e$ and $\zeta\in Z(G_e)$. Hence $\zeta$ is the right translation by $e$ in $(Q,\circ)$. It follows that $\mathcal E(\mathcal Q(G,\zeta)) = (G,\zeta)$.

Conversely, let $(Q,\cdot)$ be a connected quandle and let $\mathcal E(Q,\cdot)=(\rmlt{Q,\cdot},R_e)$ be the corresponding quandle envelope. Then, in $\mathcal Q(\mathcal E(Q,\cdot))$, we calculate $x\circ y = x^{R_e^{\,\widehat y}} = x^{R_y} = x\cdot y$. It follows that $\mathcal Q(\mathcal E(Q,\cdot)) = (Q,\cdot)$.
\end{proof}

\begin{example}
Let $K$ be a knot, $G_K$ its knot group, and $Q_K$ its knot quandle. Then $G_K$ acts transitively on the underlying set of $Q_K$, and the stabilizer of a fixed element $e\in Q$ is the peripheral subgroup $H_K$. Since $H_K\simeq\Z\times\Z$, the meridian $m$ is central in the stabilizer, and it follows from Wirtinger's presentation of $G_K$ that $G_K=\gen{ m^{G_K}}$. We proved that $(G_K,m)$ is a quandle envelope. The knot quandle $Q_K$ is isomorphic to $\mathcal Q(G_K,m)$. See \cite[Section 16]{Joy} or \cite[Section 6]{Mat} for details.
\end{example}

We conclude this section by solving the isomorphism problem and describing the automorphism group of connected quandles under the canonical correspondence.
We start with a useful characterization of isomorphisms.

\begin{lemma}\label{Lm:CanIso}
Let $(G,\zeta)$, $(K,\xi)$ be quandle envelopes on a set $Q$ with a fixed element $e\in Q$, and let
\begin{itemize}
\item $A$ be the set of all quandle isomorphisms $\varphi:\mathcal Q(G,\zeta)\to\mathcal Q(K,\xi)$ such that $e^\varphi=e$;
\item $B$ be the set of all permutations $\varphi$ of $Q$ such that $e^\varphi=e$, $\zeta^\varphi = \xi$ and $G^\varphi = K$;
\item $C$ be the set of all group isomorphisms $\psi:G\to K$ such that $\zeta^\psi = \xi$ and $G_e^{\,\psi} = K_e$.
\end{itemize}
Then $A=B$ and $\varphi\mapsto \conj{\varphi}$ is a bijection from $A=B$ to $C$.
\end{lemma}

\begin{proof}
Let $f$ denote the mapping $\varphi\mapsto \conj{\varphi}$ defined on $B$. We show that $A\subseteq B$, that $f$ maps $B$ into $C$, and we construct a mapping $g:C\to A\subseteq  B$ such that $fg$ is the identity mapping on $B$ and $gf$ is the identity mapping on $C$. This will prove the result.

Let $\mathcal Q(G,\zeta) = (Q,\circ)$, where $x\circ y = x^{\zeta^{\widehat y}}$ for some $\widehat{y}\in G$ satisfying $e^{\widehat y}=y$, and $\mathcal Q(K,\xi) = (Q,*)$, where $x*y = x^{\xi^{\overline y}}$ for some $\overline{y}\in K$ such that $e^{\overline y}=y$. For a permutation $\varphi$ of $Q$, the following universally quantified identities are equivalent:
\begin{displaymath}
    (x\circ y)^\varphi = (x^\varphi)*(y^\varphi),\qquad
    (x^{\zeta^{\widehat y}})^\varphi = (x^\varphi)^{\xi^{\overline{y^\varphi}}},\qquad
    x^{\varphi^{-1}\zeta^{\widehat y}\varphi} = x^{\xi^{\overline{y^\varphi}}}.
\end{displaymath}
Hence $\varphi$ is an isomorphism $(Q,\circ)\to (Q,*)$ if and only if
\begin{displaymath}
    (\zeta^{\widehat y})^\varphi = \xi^{\overline{y^\varphi}}.
\end{displaymath}
We will use this fact freely, as well as Lemma \ref{Lm:Q}.

$(A\subseteq B)$: We need to show $\zeta^\varphi=\xi$ and $G^\varphi = K$. Since $e^\varphi=e$, we have $\zeta^\varphi = (\zeta^{\widehat e})^\varphi = \xi^{\overline{e^\varphi}} = \xi^{\overline e} = \xi$. To prove $G^\varphi\subseteq K$, note that $G=\gen{ \zeta^G}$, pick $\alpha\in G$, and calculate $(\zeta^{\alpha})^\varphi = (\zeta^{\widehat {e^\alpha}})^\varphi = \xi^{\overline{e^{\alpha\varphi}}}\in K$. For the other inclusion $K\subseteq G^\varphi$, note that $K = \gen{ \xi^{K}}$, pick $\beta\in K$, find $\alpha\in G$ such that $e^\beta = e^{\alpha\varphi}$ by transitivity of $G$, and calculate $\xi^{\beta} = \xi^{\overline{e^\beta}} = \xi^{\overline{e^{\alpha\varphi}}} = (\zeta^{\widehat{e^\alpha}})^\varphi\in G^\varphi$.

$(f:B\to C)$: For $\varphi\in B$ let $\psi=\varphi^f = \phi_\varphi$ be the conjugation by $\varphi$. Since $G^\varphi = K$, we see that $\psi$ is an isomorphism $G\to K$. Clearly $\zeta^\psi = \zeta^\varphi = \xi$. To verify $G_e^{\,\psi} = K_e$, let $\alpha\in G_e$ and calculate $e^{\alpha^\psi} = e^{\alpha^\varphi} = e^{\varphi^{-1}\alpha\varphi} = e$, so $\alpha^\psi\in K_e$.

$(g:C\to A)$: For $\psi\in C$, define $\varphi = \psi^g$ by
\begin{displaymath}
    x^\varphi = e^{\widehat x^{\,\psi}}
\end{displaymath}
for every $x\in Q$. We show that $\varphi$ is an isomorphism $(Q,\circ)\to(Q,*)$ that fixes~$e$. The second condition follows immediately from $e^\varphi = e^{\widehat e^{\,\psi}} = e$, because $\widehat e\in G_e$ and $G_e^{\,\psi} = K_e$.
Let us observe two facts. First, if $\alpha$, $\beta\in G$, then
\begin{displaymath}
    e^{\alpha^\psi} = e^{\beta^\psi}\ \Leftrightarrow\  e^{\beta^\psi(\alpha^\psi)^{-1}}=e\ \Leftrightarrow\  (\beta\alpha^{-1})^\psi\in K_e\ \Leftrightarrow\  \beta\alpha^{-1}\in G_e\ \Leftrightarrow\  e^\alpha = e^\beta,
\end{displaymath}
hence $\varphi$ is a bijection. Second, for any $x\in Q$ and $\alpha\in G$ we have $e^{\widehat{x^\alpha}} = x^\alpha = e^{\widehat x\alpha}$. Combining the two observations, we see that
\begin{equation}\label{Eq:AuxPsi}
    e^{\widehat{x^\alpha}^\psi} = e^{(\widehat x\alpha)^\psi}.
\end{equation}
For $x$, $y\in Q$, we then have
\begin{align*}
    (x\circ y)^\varphi
    &= e^{\widehat{x\circ y}^{\,\psi}} = e^{\widehat{x\zeta^{\widehat y}}^\psi}
    = e^{(\widehat x \zeta^{\widehat y})^\psi}
    = e^{\widehat x{\,^\psi}(\zeta^{\widehat y})^\psi} \\
    & =(x^\varphi)^{(\zeta^{\widehat y})^\psi}
    = (x^\varphi)^{(\zeta^\psi)^{(\widehat y^{\,\psi})}} = (x^\varphi)^{\xi^{(\widehat y^{\,\psi})}} = (x^\varphi)^{\xi^{\overline{y^\varphi}}}
    = x^\varphi * y^\varphi,
\end{align*}
where in the penultimate step we have used $e^{\widehat y^{\,\psi}} = y^\varphi$.

($fg=\mathrm{id}$): For $\varphi\in B$ and $x\in Q$ we have
\begin{displaymath}
    x^{\varphi^{fg}}=x^{(\varphi^f)^g} = e^{\widehat x^{\,(\varphi^f)}} = e^{\widehat x^{\,\varphi}} = e^{\varphi^{-1}\widehat x\varphi} = e^{\widehat x\varphi} = x^\varphi.
\end{displaymath}

($gf=\mathrm{id}$): For $\psi\in C$ and $\alpha\in G$, we would like to show that $\alpha^{\psi^{gf}} = \alpha^{(\psi^g)^f} = \alpha^{\psi^g}$ is equal to $\alpha^\psi$. Let $x\in Q$, set $u=x^{(\psi^g)^{-1}}$ for brevity, and keeping \eqref{Eq:AuxPsi} in mind, calculate
\begin{displaymath}
    x^{\alpha^{\psi^g}} = x^{(\psi^g)^{-1}\alpha\psi^g} = (u^\alpha)^{\psi^g} = e^{\widehat{u^\alpha}^\psi}
        = e^{(\widehat u\alpha)^\psi} = e^{\widehat u^{\,\psi}\alpha^\psi} = (u^{\psi^g})^{\alpha^\psi} = x^{\alpha^\psi}.
\end{displaymath}
\end{proof}

A solution to the isomorphism problem now easily follows.

\begin{theorem}\label{Th:CanIso}
Let $(G,\zeta)$, $(K, \xi)$ be quandle envelopes on a set $Q$ with a fixed element $e\in Q$. Then the following conditions are equivalent:
\begin{enumerate}
\item[(i)] $\mathcal Q(G,\zeta)\simeq \mathcal Q(K,\xi)$.
\item[(ii)] There is a permutation $\varphi$ of $Q$ such that $e^\varphi=e$, $\zeta^\varphi = \xi$ and $G^\varphi = K$.
\item[(iii)] There is an isomorphism $\psi:G\to K$ such that $\zeta^\psi = \xi$ and $G_e^{\,\psi} = K_e$.
\end{enumerate}
\end{theorem}

\begin{proof}
Let $\rho:\mathcal Q(G,\zeta)\to\mathcal Q(K,\xi)$ be an isomorphism, and let $\alpha\in K$ be such that $e^{\rho\alpha}=e$. Since $\alpha\in K = \rmlt{\mathcal Q(K,\xi)}\le \aut{\mathcal Q(K,\xi)}$ by Theorem \ref{Th:Canonical}, the permutation $\varphi=\rho\alpha$ is also an isomorphism $\mathcal Q(G,\zeta)\to\mathcal Q(K,\xi)$ and it satisfies $e^\varphi=e$. The rest follows from Lemma \ref{Lm:CanIso}.
\end{proof}


Recall that two permutation groups acting on a set $Q$ are said to be \emph{equivalent} if they are conjugate in the symmetric group $S_Q$. Theorem \ref{Th:CanIso} shows that if the connected quandles $\mathcal Q(G,\zeta)$, $\mathcal Q(K,\xi)$ are isomorphic, then the transitive groups $G$, $K$ are equivalent, and the permutations $\zeta$, $\xi$ have the same cycle structures. While enumerating connected quandles of order $n$, it therefore suffices to investigate transitive groups of degree $n$ up to equivalence, which is the usual way transitive groups are cataloged in computational packages. The following result solves the isomorphism problem for a fixed transitive group $G$.

\begin{corollary}\label{Cr:CanIso}
Let $(G,\zeta)$, $(G,\xi)$ be quandle envelopes on a set $Q$ with a fixed element $e\in Q$. Then $\mathcal Q(G,\zeta)$ is isomorphic to $\mathcal Q(G,\xi)$ if and only if $\zeta$ and $\xi$ are conjugate in $N_{(\sym{Q})_e}(G)$, the normalizer of $G$ in the stabilizer of $e$ in the symmetric group $\sym{Q}$.
\end{corollary}

Another application of Lemma \ref{Lm:CanIso} reveals the structure of the automorphism group of a connected quandle in terms of its right multiplication group.
For a group $G$, a subgroup $H\le G$ and an element $x\in G$ we let
\begin{displaymath}
    \aut G_{x,H}=\{\psi\in\aut G\,:\,x^\psi=x,\ H^\psi=H\}\le\aut{G}.
\end{displaymath}

\begin{proposition}\label{Pr:Aut}
Let $Q=(Q,\cdot)$ be a connected quandle, $e\in Q$, and let $G=\rmlt Q$. Then $\aut{Q}$ is isomorphic to $\left(G\rtimes\aut G_{R_e,G_e}\right)/\{(\alpha,\phi_\alpha^{-1})\,:\,\alpha\in G_e\}.$
\end{proposition}

\begin{proof}
By Theorem \ref{Th:Canonical}, we have $(Q,\cdot)=\mathcal Q(G,R_e)$. According to Lemma \ref{Lm:CanIso}, $\varphi\mapsto \conj{\varphi}$ is a bijection between $\aut{Q}_e$ and $\aut{G}_{R_e,G_e}$, which is easily seen to be a homomorphism.
Define $f:G\rtimes\aut{Q}_e\to\aut{Q}$ by $(\alpha,\varphi)^f=\alpha\varphi$. This is a homomorphism, since
\begin{displaymath}
    (\alpha,\varphi)^f(\beta,\psi)^f=\alpha\varphi\beta\psi=\alpha\beta^{\varphi^{-1}}\varphi\psi=((\alpha,\varphi)(\beta,\psi))^f.
\end{displaymath}
Since $G$ acts transitively on $Q$, every $\psi\in\aut{Q}$ can be decomposed as $\psi=\alpha\varphi$, where $\alpha\in G$ and $\varphi\in\aut{Q}_e$. Thus $f$ is surjective. The kernel of $f$ consists of all tuples $(\alpha,\varphi)$ with $\alpha\varphi=1$, hence $\varphi=\alpha^{-1}\in G\cap\aut{Q}_e=G_e$.
\end{proof}

\section{Latin quandles}\label{Sc:Latin}

A quandle $Q$ is called \emph{latin}, if also the left translations
\begin{displaymath}
    L_x:Q\to Q,\quad y\mapsto xy
\end{displaymath}
are permutations of $Q$. Every latin quandle is connected. Indeed, given $x,y\in Q$, let $z$ be the unique element such that $xz=y$, and we have $x^{R_z}=y$.

In this section, we determine when a finite quandle in the homogenous representation is latin, and which quandle envelopes correspond to latin quandles. For more details on latin quandles we refer to \cite{Sta-latin}.

\begin{lemma}[{\cite[Theorem 4.2]{Gal-ldq}}]
Let $G$ be a group, $f\in\aut G$ and $H\le C_G(f)$. Suppose that the quandle $Q=\qhom{G,H,f}$ is finite. Then $Q$ is latin if and only if, for every $a,u\in G$, \begin{equation}\label{Eq:Latin}
(u^{-1})^fu\in H^a\text{ implies }u\in H.
\end{equation}
\end{lemma}
\begin{proof}
A finite quandle is latin if and only if every left translation is one-to-one. For $x\in G$, the following statements are then equivalent:
\begin{align*}
    &\text{$L_{xH}$ is one-to-one,}\\
    &\text{$H(xy^{-1})^fy=H(xz^{-1})^fz$ implies $Hy=Hz$},\\
    &\text{$(xy^{-1})^fyz^{-1}(zx^{-1})^f\in H$ implies $yz^{-1}\in H$},\\
    &\text{$(y^{-1})^fyz^{-1}z^f\in H^{x^f}$ implies $yz^{-1}\in H$},\\
    &\text{$((u^{-1})^fu)^{z^f}\in H^{x^f}$ implies $u\in H$},
\end{align*}
where in the last equivalence we have used the substitution $u=yz^{-1}$. Now, if every $L_{xH}$ is one-to-one, we obtain \eqref{Eq:Latin} from the last line above by taking $z=1$ and $x^f=a$. Conversely, to prove that any $L_{xH}$ is one-to-one, consider $u$, $z$ such that $((u^{-1})^fu)^{z^f}\in H^{x^f}$. Then $(u^{-1})^fu\in H^{(xz^{-1})^f}$, and we can use \eqref{Eq:Latin} to conclude that $u\in H$.
\end{proof}


\begin{proposition}\label{Pr:Latin1}
Let $Q$ be a finite homogeneous quandle, $e\in Q$, and let $G$ be a normal subgroup of $\aut Q$ that is transitive on $Q$. Then $Q$ is latin if and only if for every $\alpha\in G\smallsetminus G_e$ the commutator $[R_e,\alpha]$ has no fixed points.
\end{proposition}
\begin{proof}
Consider the homogeneous representation $Q\simeq\qhom{G,G_e,f}$ from Proposition \ref{Pr:HomogeneousRepresentation}, i.e., we have $\alpha^f=\alpha^{R_e}$ for every $\alpha\in G$. Condition \eqref{Eq:Latin} says that, for every $\alpha,\beta\in G$, if $(\alpha^{-1})^{R_e}\alpha\in G_e^{\beta}$ then $\alpha\in G_e$.
Since $G_e^{\beta}=G_{e^\beta}$ and $G$ is transitive, we can reformulate the condition as follows: for every $\alpha\in G$ and $x\in Q$, if $[R_e,\alpha]\in G_x$ then $\alpha\in G_e$. In other words, if $[R_e,\alpha]$ has a fixed point then $\alpha\in G_e$.
\end{proof}

\begin{proposition}\label{Pr:Latin2}
Let $(G,\zeta)$ be a quandle envelope with $G$ finite. Then $\mathcal Q(G,\zeta)$ is a latin quandle if and only if for every $\alpha\in G\smallsetminus G_e$ the commutator $[\zeta,\alpha]$ has no fixed points.
\end{proposition}
\begin{proof}
Using Theorem \ref{Th:Canonical}, we obtain the claim by applying Proposition \ref{Pr:Latin1} to $Q=\mathcal Q(G,\zeta)$ and $G=\rmlt Q$.
\end{proof}

\section{Connected affine quandles}\label{Sc:Affine}

Let $(A,+)$ be an abelian group. Then
\begin{displaymath}
    \aff{A,+} = \{x\mapsto c+x^f\,:\,c\in A,\,f\in\aut{A,+}\}
\end{displaymath}
is a subgroup of the symmetric group over $A$, and the elements of $\aff{A,+}$ are called \emph{affine mappings} over $(A,+)$. The set of translations
\begin{displaymath}
    \mlt{A,+} = \{x\mapsto c+x\,:\,c\in A\}
\end{displaymath}
is a subgroup of $\aff{A,+}$.

Recall from Example \ref{Ex:Affine} that $\qaff{A,f}$, where $f\in\aut A$, denotes the affine quandle $(A,*)$ with multiplication $x*y=x^f+y^{1-f}$. In $\qaff{A,f}$,
\[    x^{R_y} = x^f+y^{1-f},\qquad    x^{R_y^{-1}} = x^{f^{-1}}+y^{1-f^{-1}},\]
hence the right translations are affine mappings over $(A,+)$ and $\rmlt{\qaff{A,f}}$ is a subgroup of $\aff{A,+}$. In calculations, it is useful to remember that the group $\aff{A,+}$ is isomorphic to $(A,+)\rtimes\aut{A,+}$, the holomorph of $(A,+)$, where the mapping $x\mapsto c+x^f$ corresponds to the pair $(c,f)$.

\begin{proposition}
Let $Q=\qaff{A,f}$ be an affine quandle. Then \[\dis Q=\{x\mapsto x+c\,:\,c\in\im{1-f}\},\] hence $\dis Q$ is isomorphic to $\im{1-f}$.
\end{proposition}
\begin{proof}
Let $T=\{x\mapsto x+c\,:\,c\in\im{1-f}\}$. If we show that $\dis Q=T$, then the mapping $\varphi:\im{1-f}\to\dis Q$ which maps $c$ to the translation by $c$ is an isomorphism. Note that $T$ is closed with respect to composition. For the inclusion $\dis Q\subseteq T$, we calculate
\begin{displaymath}
    z^{R_x^{-1}R_y}=(z^{f^{-1}}+x^{1-f^{-1}})^{f}+y^{1-f}=z+x^{(1-f^{-1})f}+y^{1-f}=z+x^{f-1}+y^{1-f},
\end{displaymath}
so $z^{R_x^{-1}R_y}=z+c$ with the constant $c=(x^{-1})^{1-f}+y^{1-f}\in\im{1-f}$. The generators of $\dis Q$ are therefore in $T$, and $\dis{Q}\le T$ follows.

For the other inclusion $\dis Q\supseteq T$, given $c\in\im{1-f}$, choose $x\in A$ so that $x^{f-1}=c$, and verify that $z^{R_x^{-1}R_0}=(z^{f^{-1}}+x^{1-f^{-1}})^{f}=z+c$.
\end{proof}

\begin{corollary}\label{Cr:AffineOnto}
An affine quandle $\qaff{A,f}$ is connected if and only if $1-f$ is onto.
\end{corollary}

Consequently, if $Q$ is a connected affine quandle, then the isomorphism type of the underlying abelian group $(Q,+)$ is uniquely determined by the quandle. Indeed, $(Q,+)=\im{1-f}\simeq\dis Q$. (An analogous statement does not hold for disconnected affine quandles which can be supported by non-isomorphic groups.)

Another consequence is that a finite affine quandle is connected if and only if it is latin. A~stronger result is proved in \cite[Theorem 5.10]{CEHSY}: \emph{A finite left and right distributive quandle is connected if and only if it is latin.} Infinite connected affine quandles need not be latin, however. Indeed, in $\qaff{\Z_{p^\infty},1-p}$, the mapping $1-(1-p)=p$ is onto but not one-to-one.

We are now going to establish a characterization of connected quandles that are affine, or, equivalently, medial. Condition (iii) below provides a computationally efficient criterion for checking whether a connected quandle is affine. The crucial point is condition (iv), which is interesting in its own right and will be used in Section \ref{Sc:P}. We were not able to find the characterization of Theorem \ref{Th:Affine} in the literature.

\begin{theorem}\label{Th:Affine}
The following conditions are equivalent for a connected quandle $Q$:
\begin{enumerate}
\item[(i)] $Q$ is affine.
\item[(ii)] $Q$ is medial.
\item[(iii)] $\rmlt{Q}'$ is abelian.
\item[(iv)] There is an abelian group $A=(Q,+)$ such that $\mlt{A}\le\rmlt{Q}\le\aff{A}$.
\end{enumerate}
\end{theorem}

\begin{proof}
(i) $\Rightarrow$ (ii) $\Rightarrow$ (iii): We have already seen in Example \ref{Ex:Affine} that every affine quandle is medial. By Propositions \ref{Pr:DQConnected} and \ref{Pr:DQ2}, every connected medial quandle $Q$ has $\rmlt{Q}'=\dis{Q}$ abelian.

(iii) $\Rightarrow$ (iv): Fix $e\in Q$. Since $\rmlt{Q}'$ is abelian and transitive (by Propositions \ref{Pr:DQ} and \ref{Pr:DQConnected}), it is sharply transitive. Thus for every $y\in Q$ there is a unique $\widehat y\in\rmlt{Q}'$ such that $e^{\widehat y}=y$. Define $A=(Q,+)$ by $$x+y = x^{\widehat y}.$$ We claim that $\varphi:A\to\rmlt{Q}'$, $x\mapsto\widehat x$ is an isomorphism and hence that $A$ is an abelian group. Indeed, $\varphi$ is clearly a bijection, we have $e^{\widehat{x^{\widehat y}}} = x^{\widehat y} = e^{\widehat x\widehat y}$, thus $\widehat{x^{\widehat y}} = \widehat{x}\widehat{y}$ by sharp transitivity, and so $(x+y)^\varphi = (x^{\widehat y})^\varphi = \widehat{x^{\widehat y}} = \widehat x\widehat y = x^\varphi y^\varphi$.

Since the right translation by $y$ in $A$ is $\widehat y\in\rmlt{Q}'$, we have $\mlt{A}=\rmlt Q'\le\rmlt{Q}$. To prove that $\rmlt Q\leq\aff{A}$, it suffices to show that $R_e\in\aut A$ and $x\cdot y=x^{R_e}+y^{1-R_e}$, because then $R_y\in\aff{A}$ for every $y\in Q$. We have $R_e\in\aut A$ iff $(x+y)^{R_e} = x^{\widehat y R_e}$ is equal to $x^{R_e}+y^{R_e} = x^{R_e \widehat{y^{R_e}}}$ for every $x$, $y\in Q$, which is equivalent to $\widehat y^{R_e} = \widehat {y^{R_e}}$ for every $y\in Q$. Taking advantage of sharp transitivity, the last equality is verified by $e^{\widehat y^{R_e}} = y\cdot e = e^{\widehat {y^{R_e}}}$. We have $(Q,\cdot)=\mathcal Q(\mathcal E(Q,\cdot))$ by Theorem \ref{Th:Canonical}, and hence
\begin{displaymath}
    x\cdot y = x^{R_e^{\,\widehat y}} = x^{\widehat y^{-1}R_e\widehat y} = (x-y)^{R_e} + y = y^{1-R_e} + x^{R_e}.
\end{displaymath}

(iv) $\Rightarrow$ (i): Let $0$ be the identity element of $A=(Q,+)$. Fix $y\in Q$ and denote by $\rho_y$ the right translation by $y$ in $A$. By Theorem \ref{Th:Canonical}, we have $R_y = R_0^{\,\widehat y}$ for some $\widehat y\in\rmlt{Q}$ such that $0^{\widehat y}=y$. Since $\rmlt{Q}\le\aff{A}$, there are $c\in Q$ and $g\in\aut{A}$ such that $x^{\widehat y} = c+x^g$ for every $x\in Q$. But $c=0^{\widehat y}=y$, so $x^{\widehat y} = y+x^g$ and $\widehat y = g\rho_y$. Since $\mlt{A}\le\rmlt{Q}$, we have $g=\widehat y\rho_y^{-1}\in\rmlt{Q}$. Hence $g\in\rmlt Q_0$, and since $R_0\in Z(\rmlt{Q}_0)$, we obtain $gR_0=R_0g$. Since $0^{R_0} = 0$ by idempotence, we have not only $R_0\in\aff{A}$ but in fact $R_0\in\aut{A}$. Using all these facts, we calculate
\begin{displaymath}
    x\cdot y = x^{R_0^{\,\widehat y}} = x^{\widehat y^{-1}R_0\widehat y} = x^{\rho_y^{-1}g^{-1}R_0 g \rho_y} = x^{\rho_y^{-1}R_0\rho_y} = (x-y)^{R_0}+y= x^{R_0}+y^{1-R_0}
\end{displaymath}
for every $x$, $y\in Q$, proving that $(Q,\cdot) = \qaff{A,R_0}$ is an affine quandle.
\end{proof}

\begin{corollary}\label{Cr:Affine}
The following conditions are equivalent for a quandle envelope $(G,\zeta)$:
\begin{enumerate}
\item[(i)] $\mathcal Q(G,\zeta)$ is affine.
\item[(ii)] $\mathcal Q(G,\zeta)$ is medial.
\item[(iii)] $G$ is metabelian.
\item[(iv)] There is an abelian group $A$ such that $\mlt{A}\le G\le\aff{A}$.
\end{enumerate}
\end{corollary}

Theorem \ref{Th:Affine} is related to the Toyoda-Bruck theorem \cite{Bru} which states that medial quasigroups are affine.

It is not hard to check that two connected affine quandles $\qaff{A,f}$, $\qaff{A,g}$ are isomorphic if and only if $f$ and $g$ are conjugate in $\aut{A}$. (See \cite[Lemma 1.33]{AG} for a proof, and \cite{Hou} for a generalization that includes disconnected quandles.) Therefore, to enumerate connected affine quandles with $n$ elements up to isomorphism, it suffices to consider abelian groups of order $n$ up to isomorphism, and for each group $A$ all automorphisms $f\in\aut A$ such that $1-f$ is also an automorphism, up to conjugation in $\aut{A}$.

\begin{example}
Let us enumerate connected affine quandles of prime size $p$. We can assume that $A=\mathbb Z_p$ and consider all $f\in\aut{\mathbb Z_p}\simeq\mathbb Z_p^*$ such that $1-f\ne 0$, that is, $f\ne 1$. Since $\aut{\mathbb Z_p}$ is abelian, conjugacy plays no role, and we obtain $p-2$ connected affine quandles with $p$ elements.
\end{example}

An enumeration of small affine quandles has been completed by Hou in \cite{Hou}. It turns out that the function counting affine quandles of size $n$ up to isomorphism is multiplicative (in the number-theoretic sense), hence one can focus on prime powers. Hou found explicit formulas for the number of affine quandles (and connected affine quandles) for any prime power $p^k$ with $1\le k\le 4$. See \cite[equations (4.1) and (4.2)]{Hou} for the formulas, \cite[Table 1]{Hou} for the complete list of affine quandles, and also the values $a(n)$ in our Table \ref{Tab:Enum}. For example, on $p^2$ elements, there are precisely $2p^2-3p-1$ connected affine quandles, of which $p^2-2p$ are based on $A=\Z_{p^2}$ and $p^2-p-1$ on $A=\Z_p\times\Z_p$. As we shall see in Theorems \ref{Th:p} and \ref{Th:p^2}, all connected quandles with $p$ or $p^2$ elements are affine.

\section{Enumerating small connected quandles}\label{Sc:Small}

Suppose that we wish to enumerate all connected quandles of order $n$ up to isomorphism. By Theorems \ref{Th:Canonical} and \ref{Th:CanIso}, it suffices to fix a set $Q$ of size $n$, an element $e\in Q$, and consider all quandle envelopes $(G,\zeta)$ on $Q$ (with respect to $e$), where the transitive groups $G$ are taken up to equivalence. The corresponding quandles $\mathcal Q(G,\zeta)$ then account for all connected quandles of order $n$ up to isomorphism, possibly with repetitions.

Moreover, since $\mathcal E(\mathcal Q(G,\zeta)) = (G,\zeta)$ by Theorem \ref{Th:Canonical}, we see that $G = \rmlt{\mathcal Q(G,\zeta)}$. Propositions \ref{Pr:DQ} and \ref{Pr:DQConnected} then imply that it suffices to consider transitive groups $G$ for which $G'$ is also transitive and $G/G'$ is cyclic. This disqualifies many transitive groups. The conditions $\zeta\in Z(G_e)$ and $\langle \zeta^G\rangle = G$ disqualify many other transitive groups, for instance the symmetric groups in their natural actions.

Corollary \ref{Cr:CanIso} can be used to avoid isomorphic copies. But it appears to be computationally easier to allow isomorphic copies and to filter them later with a direct isomorphism check, rather than verifying whether $\zeta$, $\xi$ are conjugate in $N_{(\sym{Q})_e}(G)$.

Here is the resulting algorithm for a given size $n$.

\begin{algorithm}\label{Alg:Enum}\ \\
\texttt{quandles $\leftarrow \emptyset$}\\
\texttt{for each $G$ in the set of transitive groups on $\{1,\dots,n\}$ up to equivalence do}\\
\texttt{\phantom{xxx}if $G'$ is transitive and $G/G'$ is cyclic then}\\
\texttt{\phantom{xxxxxx}qG $\leftarrow \emptyset$}\\
\texttt{\phantom{xxxxxx}for each $\zeta$ in $Z(G_1)$ such that $\gen{ \zeta^G }=G$ do}\\
\texttt{\phantom{xxxxxxxxx}qG $\leftarrow$ qG $\cup\ \{\mathcal Q(G,\zeta)\}$}\\
\texttt{\phantom{xxxxxx}qG $\leftarrow$ qG filtered up to isomorphism}\\
\texttt{\phantom{xxxxxx}quandles $\leftarrow$ quandles $\cup$ qG}\\
\texttt{return quandles}
\end{algorithm}

We have implemented the algorithm in \textsf{GAP} \cite{GAP4}. The source code and the output of the search are available on the website of the third author. The isomorphism check is based on the methods of the \textsf{LOOPS} \cite{LOOPS} package for \textsf{GAP}. The current version of \textsf{GAP} contains a library of transitive groups up to degree $30$. An extension up to degree $47$, except for degree $32$, can be obtained from one of the authors \cite{Hul}. The $2,801,324$ transitive groups of degree $32$ can be obtained from Derek Holt \cite{CanHol}. On an Intel Core i5-2520M 2.5GHz processor, the search for all connected quandles of order $1\le n\le 47$ with $n\neq32$ takes only several minutes, and the order $n=32$ takes about an hour.

In \cite{Ven}, Vendramin presented a similar algorithm, also based on a homogeneous representation, but he was not aware of Theorems \ref{Th:Canonical} and \ref{Th:CanIso}. He therefore had to deal with many more transitive groups, filter out quandles that were not connected, and also filter more quandles up to isomorphism, resulting in a much longer computation time (on the order of weeks).

\begin{table}[ht]
$$
\begin{array}{r|rrrrrrrrrrrrrrrr}
n &      1&2&3&4&5&6&7&8&9&10&11&12&13&14&15&16\\\hline
q(n)&    1&0&1&1&3&2&5&3&8& 1& 9&10&11& 0& 7& 9\\
\ell(n)& 1&0&1&1&3&0&5&2&8& 0& 9& 1&11& 0& 5& 9\\
a(n)&    1&0&1&1&3&0&5&2&8& 0& 9& 1&11& 0& 3& 9\\\\
n &      17&18&19&20&21&22&23&24&25&26&27&28&29&30&31&32\\\hline
q(n)&    15&12&17&10& 9& 0&21&42&34& 0&65&13&27&24&29&17\\
\ell(n)& 15& 0&17& 3& 7& 0&21& 2&34& 0&62& 7&27& 0&29& 8\\
a(n)&    15& 0&17& 3& 5& 0&21& 2&34& 0&30& 5&27& 0&29& 8\\\\
n &      33&34&35&36&37&38&39&40&41&42&43&44&45&46&47& \\\hline
q(n)&    11& 0&15&73&35& 0&13&33&39&26&41& 9&45& 0&45& \\
\ell(n)& 11& 0&15& 9&35& 0&13& 6&39& 0&41& 9&36& 0&45& \\
a(n)&     9& 0&15& 8&35& 0&11& 6&39& 0&41& 9&24& 0&45& \\
\end{array}
$$
\caption{The numbers $q(n)$ of connected quandles, $\ell(n)$ of latin quandles, and $a(n)$ of connected affine quandles of size $n\le 47$ up to isomorphism.}
\label{Tab:Enum}
\end{table}

Table \ref{Tab:Enum} shows the numbers $q(n)$ of connected quandles, $\ell(n)$ of latin quandles, and $a(n)$ of connected affine quandles of size $n\le 47$ up to isomorphism. Latin quandles are recognized by a direct test whether all left translations are permutations. Affine quandles are recognized by checking whether $G'$ is abelian, using Corollary \ref{Cr:Affine}. Note that Corollary \ref{Cr:AffineOnto} implies  $a(n)\leq \ell(n)\leq q(n)$. As we shall see, $q(p)=a(p)$ and $q(p^2)=a(p^2)$ for every prime $p$ (Theorems \ref{Th:p} and \ref{Th:p^2}), and $q(2p)=0$ for every prime $p>5$ (Theorem \ref{Th:2p}). Stein's theorem \cite[Theorem 9.9]{Ste} forces $\ell(4k+2)=0$.

The numbers $q(n)$ agree with those calculated by Vendramin in \cite{Ven}, and the numbers $a(n)$ agree with the enumeration results of Hou \cite{Hou}, as discussed at the end of Section~\ref{Sc:Affine}. 

\medskip

We conclude this section by providing examples of infinite sequences of connected quandles. The first source of examples is combinatorial, resulting from multi-transitivity of the symmetric and alternating groups.

\begin{example}\label{Ex:Small1}
For $n\ge 2$ let $G=S_n$ act on $2$-element subsets of $\{1,\dots,n\}$, let $e=\{1,2\}$ and $\zeta=(1,2)$. Then $\zeta\in Z(G_e)$ and $\gen{ \zeta^G } = G$, since all transpositions are conjugate to $\zeta$ in $S_n$. Thus $\mathcal Q(G,\zeta)$ is a connected quandle of order $\binom{n}{2}$.
\end{example}

\begin{example}\label{Ex:Small2}
For $n\ge 2$ let $G=S_n$ act on $n$-cycles by conjugation, let $e=(1,\dots,n)$ and $\zeta=(1,\dots,n)$. Since the orbit of $e$ consists of all $n$-cycles, we see that $|G_e|=n$ and $G_e=Z(G_e)=\gen{ \zeta }$, so certainly $\zeta\in Z(G_e)$. Furthermore, $\gen{\zeta^G}$ generates $S_n$ if $n$ is even (and $A_n$ if $n$ is odd). Therefore, if $n$ is even then $\mathcal Q(G,\zeta)$ is a connected quandle of order $(n-1)!$.
\end{example}

\begin{example}\label{Ex:Small3}
For $n\ge 3$ let $G=S_n$ act on $(n-2)$-tuples of distinct elements pointwise, let $e$ be the $(n-2)$-tuple $(1,\dots,n-2)$, and let $\zeta=(n-1,n)$. Then we obviously have $G_e=Z(G_e)=\gen{\zeta}$, so $\zeta\in Z(G_e)$, and $\gen{\zeta^G} = G$. Thus $\mathcal Q(G,\zeta)$ is a connected quandle of order $n!/2$.
\end{example}

\begin{example}
For $n\ge 4$ let $G=A_n$ act on $(n-3)$-tuples of distinct elements pointwise, let $e$ be the $(n-3)$-tuple $(1,\dots,n-3)$, and let $\zeta=(n-2,n-1,n)$. Since $|G_e|=6/2=3$ (because $G=A_n$, rather than $G=S_n$), we have $G_e=Z(G_e)=\gen{\zeta}$, so $\zeta\in Z(G_e)$. As $A_n$ is generated by $3$-cycles, we also have $\gen{\zeta^G} = G$. Thus $\mathcal Q(G,\zeta)$ is a connected quandle of order $n!/6$.
\end{example}

There are also geometric constructions, as illustrated by the following examples:

\begin{example}\label{Ex:SL}
For a prime power $q$, let $G=\SL_2(q)$ act (on the right) on $Q$, the set of all non-zero vectors in the plane $(\F_q)^2$. Let $e=(1,0)$. A quick calculation shows that $G_e=\{M_a\,:\,a\in\F_q\}$, where $M_a=\left(\begin{smallmatrix}1&0\\a&1\end{smallmatrix}\right)$. Let $\zeta = M_1$. Since $M_aM_b=M_{a+b}$, we have $G_e\simeq (\F_q,+)$, so $\zeta\in Z(G_e)=G_e$. We claim that $\gen{\zeta^G} = G$.

First, it is easy to check that $M_a$ is conjugate to $\zeta$ in $G$ if and only if $a$ is a square in $\F_q$. If $q$ is even then $\F_q^*$ has odd order and thus every element of $\F_q$ is a square, so $G_e\le\gen{\zeta^G}$. When $q=p^k$ is odd then $\F_q^*$ contains $|\F_q^*|/2=(q-1)/2$ squares, so $|G_e\cap\langle \zeta^G\rangle| \ge (q-1)/2 + 1 > q/3$, and Lagrange's Theorem then implies that $G_e\le\gen{\zeta^G}$ again.

Since $G_e\le\gen{\zeta^G}$, we establish $\gen{\zeta^G}=G$ by proving that $\gen{\zeta^G}$ acts transitively on $Q$. Given $(x,y)\in Q$ with $y\ne 0$, we have $(x,y) = eDM_{-y}D^{-1}$ with $D=\left(\begin{smallmatrix}0&1\\-1&d\end{smallmatrix}\right)$, $d=(1-x)y^{-1}$. In particular, $(0,1)\in e^{\gen{\zeta^G}}$, and given $(x,0)\in Q$, we obtain $(x,0) = (0,1)EM_xE^{-1}$ with $E=\left(\begin{smallmatrix}1&x^{-1}\\0&1\end{smallmatrix}\right)$.
Hence $\gen{\zeta^G}=G$, and thus $\mathcal Q(G,\zeta)$ is a connected quandle of order $q^2-1$.
\end{example}

\begin{example}\label{Ex:PSL}
For a prime power $q$, let $G=\PSL_3(q)$ act on $Q$, the set of all two-element subsets of the projective plane $\mathbb P^2(\F_q)$. This is a transitive action, because the natural action of $G$ on $\mathbb P^2(\F_q)$ is 2-transitive. Pick a two-element subset $e=\{e_1,e_2\}$ arbitrarily, and consider matrices with respect to the basis $(e_1,e_2,e_3)$, with an arbitrary completion by $e_3$. Clearly, $G_e=\{M_{a,b},N_{a,b}\,:\,a,b\in\F_q\}$, where $$M_{a,b}=\left(\begin{smallmatrix}1&0&0\\0&1&0\\a&b&1\end{smallmatrix}\right),\qquad N_{a,b}=\left(\begin{smallmatrix}0&1&0\\1&0&0\\a&b&-1\end{smallmatrix}\right).$$
A quick calculation shows that $\zeta=M_{a,-a}\in Z(G_e)$ for every $a\in\F_q$. Since $G$ is a simple group, we obtain for free that the normal subgroup $\gen{\zeta^G}$ is equal to $G$ (unless $a=0$). Thus $\mathcal Q(G,\zeta)$ is a connected quandle of order $|Q|=(q^2+q+1)(q^2+q)/2$.
\end{example}

\begin{example}\label{Ex:Small5}
The group $G$ of rotations of a Platonic solid (see~\cite[p.136]{coxeter73}) acts on faces. Let $e$ be a face.
\begin{itemize}
\item Tetrahedron: We have $G=A_4$ acting on $4$ points (faces), and with $\zeta$ a generator of $G_e\simeq\mathbb Z_3$ we get $\gen{ \zeta^G} = G$. Thus $\mathcal Q(G,\zeta)$ is a connected quandle of order $4$. Since $A_4$ is metabelian, Theorem \ref{Th:Affine} implies that $\mathcal Q(G,\zeta)$ is affine.
\item Cube: We have $G=S_4$ acting on $6$ points, and with $\zeta$ a generator of $G_e\simeq\mathbb Z_4$ we get $\gen{ \zeta^G} = G$. Thus $\mathcal Q(G,\zeta)$ is a connected quandle of order $6$.
\item Octahedron: We have $G=S_4$ acting on $8$ points, and $G_e\simeq\mathbb Z_3$. Since $3$-cycles do not generate $S_4$, no choice of $\zeta\in G_e$ yields a connected quandle $\mathcal Q(G,\zeta)$.
\item Dodecahedron:	We have $G=A_5$ acting on $12$ points, and with $\zeta$ a generator of $G_e\simeq\mathbb Z_5$ we get $\gen{ \zeta^G} = G$. Thus $\mathcal Q(G,\zeta)$ is a connected quandle of order $12$.
\item Icosahedron:	We obtain $G=A_5$ acting on $20$ points, and with $\zeta$ a generator of $G_e\simeq\mathbb Z_3$ we get $\gen{ \zeta^G} = G$. Thus $\mathcal Q(G,\zeta)$ is a connected quandle of order $20$.
\end{itemize}
\end{example}

There are algebraic constructions where the quandle envelope is not obvious. For example, the following construction of connected quandles of size $3n$, extending an affine quandle $\qaff{A,-1}$ by $\qaff{\Z_3,-1}$, presented by Clark et al. \cite{CEHSY}, inspired by Galkin \cite{Gal-small}.

\begin{example}\label{Ex:G}
Let $A$ be an abelian group and $u\in A$. We define $\mu$, $\tau:\Z_3\to A$ by $0^\mu=2$, $1^\mu=2^\mu=-1$ and $0^\tau=1^\tau=0$, $2^\tau=u$, and we define a binary operation on $\Z_3\times A$ by
\begin{displaymath}
    (x,a)\circ(y,b)=(-x-y,-a+(x-y)^\mu b+(x-y)^\tau).
\end{displaymath}
Then $\qgal{A,u}=(\Z_3\times A,\circ)$ is a connected quandle, called the \emph{Galkin quandle corresponding to the pointed group $(A,u)$}. It is affine iff $3A=0$. It is latin iff $|A|$ is odd. Two Galkin quandles are isomorphic iff the corresponding pointed groups are isomorphic. See \cite{CEHSY} for details.
\end{example}

\begin{table}[ht]
\begin{displaymath}
\begin{array}{rllll}
\text{size} & \rmlt Q & \text{construction} & \text{properties} \\\hline\hline
6 & S_4 & \text{\ref{Ex:Small1} or $\qgal{\Z_2,0}$}& \\
6 & S_4 & \text{\ref{Ex:Small2} or \ref{Ex:Small5} or $\qgal{\Z_2,1}$} & \\\hline
8 & \SL_2(3) & \text{\ref{Ex:SL}} & \\\hline
10 & S_5 & \text{\ref{Ex:Small1}} & \text{simple}\\\hline
12 & S_4 & \text{\ref{Ex:Small3}} & \\
12 & A_5 & \text{\ref{Ex:Small5}} & \text{simple}\\
12 & A_4\rtimes\Z_4 & \qhom{A_4,1,(1,2,3,4)}& \\
12 & (\Z_3^2\rtimes Q_8)\rtimes\Z_3 & & \\
12 & (\Z_4^2\rtimes\Z_3)\rtimes\Z_2 & \qgal{\Z_4,0} &\\
12 & (\Z_4^2\rtimes\Z_3)\rtimes\Z_2 & \qgal{\Z_4,1} &\\
12 & (\Z_4^2\rtimes\Z_3)\rtimes\Z_2 & \qgal{\Z_4,2} &\\
12 & (\Z_2^4\rtimes\Z_3)\rtimes\Z_2 & \qgal{\Z_2^2,(0,0)} &\\
12 & (\Z_2^4\rtimes\Z_3)\rtimes\Z_2 & \qgal{\Z_2^2,(1,1)}& \\\hline
15 & (\Z_5^2\rtimes\Z_3)\rtimes\Z_2 & \qgal{\Z_5,0} & \text{latin}\\
15 & (\Z_5^2\rtimes\Z_3)\rtimes\Z_2 & \qgal{\Z_5,1}& \text{latin}\\
15 & S_6 & \text{\ref{Ex:Small1}} & \text{simple}\\
15 & \SL_2(4) & \text{\ref{Ex:SL}} & \text{simple}\\\hline
\vdots & & & \\\hline
21 & (\Z_7^2\rtimes\Z_3)\rtimes\Z_2 & \qgal{\Z_7,0} &\text{latin}\\
21 & (\Z_7^2\rtimes\Z_3)\rtimes\Z_2 & \qgal{\Z_7,1} &\text{latin}\\
21 & S_7 & \text{\ref{Ex:Small1}} & \text{simple}\\
21 & \PSL_3(2) & \text{\ref{Ex:PSL}} & \text{simple}\\\hline
\vdots & & & \\\hline
33 & (\Z_{11}^2\rtimes\Z_3)\rtimes\Z_2 & \qgal{\Z_{11},0} &\text{latin}\\
33 & (\Z_{11}^2\rtimes\Z_3)\rtimes\Z_2 & \qgal{\Z_{11},1} & \text{latin}\\
\end{array}
\end{displaymath}
\caption{All connected non-affine quandles of certain orders.}\label{Tab:list}
\end{table}

Table \ref{Tab:list} lists all connected non-affine quandles of orders $n\le 15$ and $n\in\{21,33\}$. In the column labeled ``construction'' we either give a reference to a numbered example which uniquely determines the quandle, or we specify how the quandle can be constructed as $\qhom{G,H,f}$ of Construction \ref{Co:Homogeneous}, or we specify how the quandle can be constructed as $\qgal{A,u}$ of Example \ref{Ex:G}.

\begin{problem}
Let $p\ge 11$ be a prime. Is it true that the only non-affine connected quandles of order $3p$ are the Galkin quandles $\qgal{\Z_p,0}$ and $\qgal{\Z_p,1}$?
\end{problem}

\section{Connected quandles of order $p$, $p^2$ and $2p$}\label{Sc:P}

First, we will show that connected quandles of prime power order have a solvable right multiplication group, using a deep result on conjugacy classes of prime power size by Kazarin \cite{Kaz}. Based on that, we give two new, conceptually simple proofs that connected quandles of prime order are affine: the first argument uses an observation about $\rmlt{Q}$ of simple quandles, the second one requires Galois' result on solvable primitive groups. The orginal proof of Etingof, Soloviev and Guralnick \cite{ESG} relies on a group-theoretical result equivalent to the one of Kazarin, too.

Then we mention the result of Gra\~na \cite{Gra} that connected quandles of prime square order are affine, and conclude with a new, shorter and purely group-theoretical proof (modulo Theorem \ref{Th:Canonical}) of the recent result of McCarron \cite{McC} that there are no connected quandles of order $2p$ with $p>5$ prime.

\begin{lemma}[{\cite[Lemma 1.29]{AG}}]\label{Pr:Equivalence}
Let $Q$ be a connected rack. For $a$, $b\in Q$ let $a\sim b$ iff $R_a=R_b$. Then $\sim$ is an equivalence relation on $Q$, and all equivalence classes of $\sim$ have the same size.
\end{lemma}
\begin{proof}
It is clear that $\sim$ is an equivalence relation. Let $[a]$, $[c]$ be two equivalence classes of $\sim$. Since $Q$ is connected, there is $\theta\in\rmlt{Q}$ such that $a^\theta=c$. If $a\sim b$ then $R_c = R_{a^\theta} = R_a^{\,\theta} = R_b^{\,\theta} = R_{b^\theta}$, thus $c\sim b^\theta$, showing that $[a]^\theta\subseteq [c]$. Since $\theta$ is one-to-one, we deduce $|[a]|\le |[c]|$. The mapping $\theta^{-1}\in\rmlt{Q}$ gives the other inequality.
\end{proof}

\begin{proposition}\label{Pr:Solvable}
Let $Q$ be a connected quandle of prime power order. Then $\rmlt Q$ is a solvable group.
\end{proposition}
\begin{proof}
Kazarin proved in \cite{Kaz} that in a group $G$, if $x\in G$ is such that $|x^G|$ is a prime power, then the subgroup $\gen{ x^G}$ is solvable.

Let $Q$ be a connected quandle of prime power order, let $G=\rmlt{Q}$ and $\zeta=R_e$ for any $e\in Q$. Note that $\zeta^G = \{R_x\,:\,x\in Q\}$. By Lemma \ref{Pr:Equivalence}, $|\zeta^G|$ is a divisor of $|Q|$, hence a prime power. Kazarin's result then implies that $\gen{\zeta^G}=G$ is solvable.
\end{proof}

Recall that a quandle $Q$ is \emph{simple} if all its congruences are trivial.

\begin{theorem}[\cite{ESG}]\label{Th:p}
Every connected quandle of prime order is affine.
\end{theorem}

\begin{proof}
Let $Q$ be the quandle in question. By Proposition \ref{Pr:Solvable}, $G=\rmlt{Q}$ is solvable. Moreover, since $G$ acts transitively on a set of prime size, it must act primitively.

\emph{Proof 1.} Consequently, the quandle $Q$ is simple, because every congruence of $Q$ is invariant under the action of $G$. An observation by Joyce \cite[Proposition 3]{Joy-simple} says that if $Q$ is simple then $G'$ is the smallest nontrivial normal subgroup in $G$. Since $G$ is solvable, we then must have $G''=1$, hence $G'$ is abelian, and so $Q$ is affine by Theorem \ref{Th:Affine}.

\emph{Proof 2.} A theorem of Galois says that a solvable primitive group acts as a subgroup of the affine group over a finite field. Theorem \ref{Th:Affine} now concludes the proof.
\end{proof}

An analogous statement holds for prime square orders, but the reason seems to be more complicated. Gra\~na's proof relies on an examination of several cases of the right multiplication group of a potential counterexample.

\begin{theorem}[\cite{Gra}]\label{Th:p^2}
Every connected quandle of prime square order is affine.
\end{theorem}

We now turn our attention to order $2p$. For every integer $n\ge 2$, Example \ref{Ex:Small1} yields a connected quandle of order $\binom{n}{2}$. With $n=4$ and $n=5$ we obtain connected quandles of order $6=2\cdot 3$ and $10=2\cdot 5$, respectively. These examples are sporadic in the sense that $\binom{n}{2}$ is equal to $2p$ for a prime $p$ if and only if $n\in\{4,5\}$.

\begin{theorem}[\cite{McC}]\label{Th:2p}
There is no connected quandle of order $2p$ for a prime $p>5$.
\end{theorem}

We conclude the paper with a new proof of Theorem \ref{Th:2p}.
Suppose that $Q$ is a connected quandle of order $2p$. Then $G=\rmlt{Q}\le\sym{2p}$, $G'$ acts transitively on $Q$ by Proposition \ref{Pr:DQConnected}, and $\gen{ \zeta^G} = G$ for some $\zeta\in Z(G_e)$ by Theorem \ref{Th:Canonical}, so, in particular, $\gen{ Z(G_e)^G} =G$. Theorem \ref{Th:2p} therefore follows from the group-theoretical Theorem \ref{thmperm} below that we prove separately.

\section{A result on transitive groups of degree $2p$}

\begin{theorem}\label{thmperm}
Let $p>5$ be a prime. There is no transitive group $G\le S_{2p}$ satisfying both of the following conditions:
\begin{enumerate}
\item[(A)] $G'$ is transitive on $\{1,\ldots,2p\}$.
\item[(B)] $\gen{Z(G_1)^G}=G$.
\end{enumerate}
\end{theorem}

We start with two general results on the center of the stabilizer of almost simple primitive groups of degree $p$ and $2p$. Both proofs are based on the explicit classification of almost simple primitive groups of degree $p$ and $2p$ \cite{shareshianmobius} (which are essentially results from~\cite{guralnickppsimple,liebecksaxl85}). In the next subsection, we prove Theorem \ref{thmperm}.

We will use repeatedly the easy fact that a nontrivial normal subgroup of a transitive group does not have fixed points.

\subsection{Almost simple primitive groups of degree $p$, $2p$}

\begin{theorem}\label{Th:primitive_p}
Let $p\geq5$ be a prime, $G\le S_p$ an almost simple primitive group, $U=G_1$ and $V\le U$ with $[U:V]\le 2$. Then $Z(V)=\gen{1}$.
\end{theorem}

An explicit classification of these groups is given in \cite[Lemma 3.1]{shareshianmobius}:

\begin{lemma}\label{Lm:AlmostSimple}
Let $p$ be a prime and $G\le S_p$ be an almost simple primitive group. Then $K=\Soc(G)$ is one of the following groups:
\begin{enumerate}
\item[(i)] $K=A_p$,
\item[(ii)] $K=\PSL_d(q)$ acting on $1$-spaces or hyperplanes of its natural projective space, $d$ is a prime and $p=(q^d-1)/(q-1)$,
\item[(iii)] $K=\PSL_2(11)$ acting on cosets of $A_5$,
\item[(iv)] $K=M_{23}$ or $K=M_{11}$.
\end{enumerate}
\end{lemma}

For case (ii) we note the following fact:

\begin{lemma}\label{pslstacen}
Let $d\ge 2$ and $q$ be a prime power such that $(d,q)\not=(2,2)$. Let $G=\aut{\PSL_d(q)}$, $U$ be the stabilizer in $G$ of a 1-dimensional subspace, $W=U\cap\PSL_d(q)$ and $V\le W$ with $[W:V]\le 2$. Then $C_U(V)=\gen{1}$.
\end{lemma}
\begin{proof}
Since the graph automorphism of $\PSL_d(q)$ swaps the stabilizers of 1-dimensional subspaces with those of hyperspaces it cannot be induced by $U$. Thus $U\le\mbox{P$\Gamma$L}_d(q)$ and elements of $U$ can be represented by pairs [field automorphism, matrix] of the form
\[
\left[
\tau,\left(\begin{array}{cc}
a&0\\
B&A\\
\end{array}\right)\right]
\]
with $a\in\F_q^*$, $B\in\F_q^{d-1}$ and $A\in\GL_{d-1}(q)$ and $\tau\in\gen{\sigma}$. Two such elements multiply as
\[
\left[\tau_1,\left(\begin{array}{cc}
a_1&0\\
B_1&A_1\\
\end{array}\right)\right]\cdot
\left[\tau_2,\left(\begin{array}{cc}
a_2&0\\
B_2&A_2\\
\end{array}\right)\right]=
\left[\tau_1\tau_2,\left(\begin{array}{cc}
a_1a_2&0\\
B_1^{\tau_2}+A_1^{\tau_2}B_2&A_1^{\tau_2}A_2\\
\end{array}\right)\right]\]

Elements of $W$ will have a trivial field automorphism part and $a\cdot \det(A)=1$, thus the $A$-part includes all of $\SL_{d-1}(q)$. If $V\not=W$ we have $V\lhd W$ of index $2$, so it has a smaller $A$-part. (If it had a smaller $B$-part, this would have to be a submodule for the natural $\SL_{d-1}(q)$-module which is irreducible.) The $A$-part cannot be smaller if $d-1\ge 3$, or if $d-1=2$ and $q\ge 4$.

In the remaining cases ($d-1=2$ and $q\in\{2,3\}$; respectively $d-1=1$) the $A$-part can be smaller by index 2. However we note by inspection that there is no $B$-part that is fixed by all $A$-parts by multiplication.

We now consider a pair of elements, the second being in $V$ and the first being in $C_U(V)$. By the multiplication formula the elements commute only if $B_1^{\tau_2}+A_1^{\tau_2}B_2=B_2^{\tau_1}+A_2^{\tau_1}B_1$. We will select elements of $V$ suitably to impose restrictions on $C_U(V)$.

If $A_1$ is not the identity we can set $A_2$ as identity, $B_2$ a vector defined over the prime field moved by $A_1$, and $\tau_2=1$ violating the equality. Similarly, if $B_1$ is nonzero (with trivial $A_1$) we can choose $B_2$ to be zero, $\tau_2=1$ and $A_2$ a matrix defined over the prime field that moves $B_1$ (we noted above such matrices always exist in $V$) to violate the equality. Finally, if $B_1$ is zero and $A_1$ the identity but $\tau_1$ nontrivial we can chose $\tau_2$ to be trivial and $B_2$ a vector moved by $\tau_1$ and violate the equation. This shows that the only element of $U$ commuting with all of $V$ is the identity.
\end{proof}

\begin{corollary}\label{pslcor}
Let $\PSL_d(q)\le G\le \aut{\PSL_d(q)}$, $U$ be the stabilizer in $G$ of a 1-dimensional subspace,  and $W\le U$ with $[U:W]\le 2$. Then $Z(W)=\gen{1}$.
\end{corollary}
\begin{proof}
As subgroups of index 2 are normal we know that there exists a subgroup $V\le W$ as specified in Lemma~\ref{pslstacen}. But then by this lemma
\[
Z(W)\le C_W(V)\le C_{\aut{PSL_d(q)}_{\text{subspace}}}(V)=\gen{1}.
\]
\end{proof}

\begin{proof}[Proof of Theorem \ref{Th:primitive_p}]
For case (i) of Lemma \ref{Lm:AlmostSimple}, we have that $U\in\{S_{p-1}$,
$A_{p-1}\}$ and so also $V\in\{S_{p-1}$, $A_{p-1}\}$, thus (as $p\ge 5$)
clearly $Z(V)=\gen{1}$. For case (ii) we get from Corollary~\ref{pslcor}
that $Z(V)=\gen{1}$. Finally for the groups in cases (iii) and (iv) an explicit
calculation in \textsf{GAP} (as $U/V$ is abelian we can find all candidates for $V$ by calculating in $U/U'$) establishes the result.
\end{proof}

Now we turn to the case $2p$.

\begin{theorem}\label{Th:primitive_2p}
Let $p>5$ be a prime and $G\le S_{2p}$ a primitive group. Then $Z(G_1)=\gen1$.
\end{theorem}

By the O'Nan-Scott theorem~\cite{liebeckpraegersaxl88}, $G$ must be almost simple. An explicit classification of these groups is given in~\cite[Theorem 4.6]{shareshianmobius}.

\begin{lemma}\label{Lm:Primitive}
Let $p$ be a prime and $G\le S_{2p}$ be a primitive group. Then $K=\Soc(G)$ is one of the following groups:
\begin{enumerate}
\item[(i)] $K=A_{2p}$,
\item[(ii)] $p=5$, $K=A_5$ acting on 2-sets,
\item[(iii)] $2p=q+1$, $q=r^{2^a}$ for an odd prime $r$, $K=\PSL_2(q)$ acting on $1$-spaces,
\item[(iv)] $p=11$, $K=M_{22}$.
\end{enumerate}
\end{lemma}

\begin{proof}[Proof of Theorem \ref{Th:primitive_2p}]
In case (i) of Lemma \ref{Lm:Primitive} we have that $G\in\{S_{2p}, A_{2p}\}$ and thus $G_1\in\{S_{2p-1},A_{2p-1}\}$ for which the statement is clearly true. Case (ii) is irrelevant here as $p=5$. Case (iii) follows from Corollary~\ref{pslcor}. Case (iv) is again done with an explicit calculation in~\textsf{GAP}.
\end{proof}

\subsection{Proof of Theorem \ref{thmperm}}

We start by discussion what block systems are afforded by~$G$.

\begin{lemma}
If $G$ is primitive, then condition \emph{(B)} is violated.
\end{lemma}
\begin{proof}
This is a direct consequence of Theorem \ref{Th:primitive_2p}.
\end{proof}

\begin{lemma}\label{blockp}
If $G$ affords a block system with blocks of size $p$, then condition \emph{(A)} is violated.
\end{lemma}
\begin{proof}
Consider a block system with two blocks of size $p$ and $\varphi\colon G\to S_2$ the action on these blocks. Then $[G:\ker\varphi]=2$, and thus $G'\le\ker\varphi$ is clearly intransitive.
\end{proof}

So it remains to check the case when $G$ has $p$ blocks of size $2$. Denote the set of blocks by $\mathcal B$, let $1\in B\in\mathcal B$. Labeling points suitably, we can assume that $B=\{1,2\}$. Let $S=G_1$ be a point stabilizer and $T=G_B$ a (setwise) block stabilizer.

Let $\varphi\colon G\to S_p$ be the action on the blocks. We set $H=\im\varphi\le S_p$ and $M=\ker\varphi$ and note that $M\le C_2^p$ is either trivial or has exactly $p$ orbits of length $2$.

\begin{lemma}\label{stabimg}
If $M\not=\gen{1}$ then $T=MS$.
\end{lemma}
\begin{proof}
If $M\not=\gen{1}$, then $M$ has orbits of length $2$. Consider $t\in T$. If $1^t\not=1$ then $1^t=2$ is in the same $M$-orbit. Thus there exists $m\in M$ such that $1^t=1^m$, thus $tm^{-1}\in S$.
\end{proof}

As $p$ is a prime, $H$ is a primitive group. By the O'Nan-Scott theorem~\cite{liebeckpraegersaxl88}, we know that $H$ is either of affine type or almost simple.

\begin{lemma}
If $H$ is almost simple, then condition \emph{(B)} is violated.
\end{lemma}
\begin{proof}
If $M\not=\gen{1}$ then by Lemma~\ref{stabimg} $S^\varphi=T^\varphi=H_1$. But then $Z(S)^\varphi\le Z(H_1)=\gen{1}$ by Theorem~\ref{Th:primitive_2p}. Thus $Z(S)\le\ker\varphi\lhd G$ and $\gen{Z(S)^G}\not=G$.

If $M=\gen{1}$ then $\varphi$ is faithful and $G\simeq H$. The point stabilizer $S\le G$ is (isomorphic to) a subgroup of the point stabilizer of $H$ of index 2. But then by Theorem~\ref{Th:primitive_p} we have that $Z(S)=\gen{1}$ and thus $\gen{Z(S)^G}\not=G$.
\end{proof}

It remains to consider the affine case, i.e. $H\le \F_p\rtimes \F_p^*$. We can label the $p$ points on which $H$ acts as $0,\ldots,p-1$, then the action of the $\F_p$-part is by addition, and that of the $\F_p^*$-part by multiplication modulo $p$. Without loss of generality assume that $T^\varphi=H_1$.
We may also assume that $H$ is not cyclic as otherwise $H'=\gen{1}$ and thus $G'\le M$ and condition (A) would be violated.

For $p=7$ an inspection of the list of transitive groups of degree 14~\cite{conwayhulpkemckay} shows that there is no group of degree 14 which fulfills (A) and (B). Thus it remains to consider $p>7$.

Let $L=S\cap M=M_1$.

\begin{lemma}\label{lsize}
If $|L|\le2$ and $p>7$ then condition \emph{(A)} is violated.
\end{lemma}
\begin{proof}
If $|L|\le 2$ then $|M|\le 4$ and $|G|$ divides $4p(p-1)$. Consider the number $n$ of $p$-Sylow subgroups of $G$. Then $n\equiv 1\pmod{p}$ and $n$ divides $4(p-1)$. Thus $n=ap+1$ with $a\in\{0,1,2,3\}$ and $b(ap+1)=4(p-1)$. If $a\not=0$ this implies that $b\in\{1,2,3,4\}$. Trying out all combinations $(a,b)$ we see that there is no solution for $a>0$, $p>7$.

So $n=1$. But a normal $p$-Sylow subgroup must have two orbits of length $p$, which as orbits of a normal subgroup form a block system for $G$. The result follows by Lemma \ref{blockp}.
\end{proof}

This in particular implies that we can assume that $M\not=\gen{1}$, thus by Lemma~\ref{stabimg} we have that $S^\varphi=H_1\le\F_p^*$. Thus there exists $b\in S$ such that $H_1=\gen{b^\varphi}$.

\begin{lemma}
$S=\gen{b}\cdot L$.
\end{lemma}
\begin{proof}
Clearly $S\ge\gen{b}\cdot L$. Consider $s\in S$. Then $s^\varphi\in H_1$, thus $s^\varphi=(b^\varphi)^x$ for a suitable $x$ and thus $sb^{-x}\in\ker\varphi\cap S=L$.
\end{proof}

We shall need a technical lemma about finite fields. For $\beta\in\F_p^*$, a subset $I\subset\F_p$ is called \emph{$\beta$-closed} if $I\beta=I$, that is $x\in I$ iff $x\beta\in I$.

\begin{lemma} \label{closedlemma}
Let $\alpha,\beta\in\F_p^*$, $\beta\neq 1$ and assume that $\emptyset\not=I\subset\F_p^*$ is $\beta$-closed. Then $I-\alpha=\left\{i-\alpha\mid i\in I\right\}$ is not $\beta$-closed.
\end{lemma}
\begin{proof}
Assume that $I-\alpha$ is $\beta$-closed and consider an arbitrary $x\in I$. Then (as $\beta$ has a finite multiplicative order) $x\beta^{-1}\in I$ and thus $x\beta^{-1}-\alpha\in I-\alpha$. But by the assumption $(x\beta^{-1}-\alpha)\beta\in I-\alpha$ and thus $(x\beta^{-1}-\alpha)\beta+\alpha=x+\alpha(1-\beta)\in I$. Thus $I$ would be closed under addition of $\alpha(1-\beta)\not=0$. But the additive order of a nonzero element in $\F_p$ is $p$, implying that $I=\F_p$, contradicting that $0\not\in I$.
\end{proof}

\begin{lemma}
If condition \emph{(A)} holds, then $Z(S)\le L\le M$.
\end{lemma}
\begin{proof}
Assume the condition holds. We show the stronger statement that $C_S(L)\le L$. For this assume to the contrary that $b^x\cdot l\in C_S(L)$ with $l\in L$ and $x$ a suitable exponent such that $b^x\not\in L$. As $L\le M$ is abelian this implies that $b^x\in C_S(L)$. Let $\beta\in\F_p^*\le H$ be such that $(b^x)^\varphi=\beta$. As $b^x\not\in L$ we know that $\beta\not=1$.

When we consider the conjugation action of $G$ on $M\le C_2^p$, note that an element of $M$ is determined uniquely by its support (that is the blocks in $\mathcal B$ whose points are moved by the element), which we consider as a subset of $\F_p$, which is the domain on which $H$ acts. An element $g\in G$ acts by conjugation on $M$ with the effect of moving the support of elements in the same way as $g^\varphi$ moves the points $\F_p$. For $b^x$ to centralize an element $a\in L$, the support $I$ of $a$ thus must be $\beta$-closed for $\beta=(b^x)^\varphi$.

By Lemma~\ref{lsize} we can assume that $|L|>2$. Thus there exists an element $a\in L$ whose support $I$ is a proper nonempty subset of $\F_p^*$. Thus there exists $\alpha\in\F_p^*$, $\alpha\not\in I$.

That means that if we conjugate $a$ with $-\alpha\in\F_p$, the resulting element $\tilde a$ has support $I-\alpha$. By assumption $0\not\in I-\alpha$, so $\tilde a\in L$. But by Lemma~\ref{closedlemma} we know that $I-\alpha$ is not $\beta$-closed, that is $\tilde a\in L$ is not centralized by $b^x$.
\end{proof}

\begin{corollary}
If $H$ is of affine type, then at least one of conditions \emph{(A), (B)} is violated.
\end{corollary}
\begin{proof}
If (A) holds, then $\gen{Z(G_1)^G}\le M\neq G$.
\end{proof}

This concludes the proof of Theorem~\ref{thmperm}.

\section*{acknowledgment}

We thank Derek Holt for the library of transitive groups of degree $32$. We also thank an anonymous referee for a number of useful comments, particularly regarding the presentation of older results.

\end{document}